\documentclass[11pt]{amsart}
\usepackage{amssymb, amstext, amscd, amsmath}
\usepackage[all]{xy}
\usepackage[notcite,notref]{showkeys}

\makeatletter
\def\@cite#1#2{{\m@th\upshape\bfseries%
[{#1\if@tempswa{\m@th\upshape\mdseries, #2}\fi}]}}
\makeatother
\theoremstyle{plain}
\newtheorem{theorem}{Theorem}[section]
\newtheorem{corollary}[theorem]{Corollary}
\newtheorem{proposition}[theorem]{Proposition}
\newtheorem{lemma}[theorem]{Lemma}

\theoremstyle{definition}
\newtheorem{definition}[theorem]{Definition}
\newtheorem{example}[theorem]{Example}
\newtheorem{remark}[theorem]{Remark}
\newtheorem*{acknow}{Acknowledgements}
\theoremstyle{remark}
\newcommand{\bbC}{{\mathbb{C}}}
\newcommand{\bbN}{{\mathbb{N}}}
\newcommand{\bbZ}{{\mathbb{Z}}}

\newcommand{\F}{{\mathcal{F}}}
\newcommand{\G}{{\mathcal{G}}}

\newcommand{\K}{{\mathcal{K}}}
\renewcommand{\L}{{\mathcal{L}}}
\newcommand{\M}{{\mathcal{M}}}
\renewcommand{\O}{{\mathcal{O}}}
\renewcommand{\S}{{\mathcal{S}}}
\newcommand{\T}{{\mathcal{T}}}

\newcommand{\fA}{{\mathfrak{A}}}
\newcommand{\fX}{{\mathfrak{X}}}

\def\gs{\sigma}
\def\ga{\alpha}
\def\be{\beta}

\def\de{\delta}

\newcommand{\foral}{\text{ for all }}

\newcommand{\alg}{\operatorname{alg}}

\newcommand{\id}{{\operatorname{id}}}
\newcommand\Span{\mathop{\rm span}}

\newcommand{\ca}{\mathrm{C}^*}
\newcommand{\sca}[1]{\left\langle#1\right\rangle} % \sca{a,b} =<a,b>
\newcommand{\lsca}[1]{\left[#1\right]}            % \lsca{a,b} =(a,b)
\newcommand{\nor}[1]{\left\Vert #1\right\Vert}    %\nor{x}=||x||

\newcommand{\sse}{\stackrel{\text{s}}{\thicksim}}
\newcommand{\sme}{\stackrel{\text{SME}}{\thicksim}}
\newcommand{\se}{\stackrel{\text{SE}}{\thicksim}}
\newcommand{\tsse}{{\stackrel{\text{SSE}}{\thicksim}}}

\begin{document}

\title{C*-algebras and Equivalences for C*-correspondences}

\author[E.T.A. Kakariadis]{Evgenios T.A. Kakariadis}
\address{Pure Math.\ Dept.\\U. Waterloo\\Waterloo, ON\; N2L--3G1\\CANADA}
\email{ekakaria@uwaterloo.ca}

\author[E.G. Katsoulis]{Elias~G.~Katsoulis}
\address{ Department of Mathematics\\University of Athens
\\ 15784 Athens \\Greece \vspace{-2ex}}
\address{\textit{Alternate address:} Department of Mathematics
\\East Carolina University\\ Greenville, NC 27858\\USA}
\email{katsoulise@ecu.edu}

\thanks{First author partially supported by the Fields Institute for Research in the Mathematical Sciences}

\maketitle

\vspace{-.5cm}
\begin{center}
{\small \emph{Dedicated to our mentor and friend Aristides Katavolos}}
\end{center}

%%%%%%%%%%%%%%%%
\begin{abstract}
We study several notions of shift equivalence for $\ca$- correspondences and the effect that these equivalences have on the corresponding Pimsner dilations. Among others, we prove that non- degenerate, regular, full $\ca$-correspondences which are shift equivalent have strong Morita equivalent Pimsner dilations. We also establish that the converse may not be true. These results settle open problems in the literature.

In the context of $\ca$-algebras, we prove that if two non-degenerate, regular, full $\ca$-correspondences are shift equivalent, then their corresponding Cuntz-Pimsner algebras are strong Morita equivalent. This generalizes results of Cuntz and Krieger and Muhly, Tomforde and Pask. As a consequence, if two subshifts of finite type are eventually conjugate, then their Cuntz-Krieger algebras are strong Morita equivalent.

Our results suggest a natural analogue of the Shift Equivalence  Problem in the context of $\ca$-correspondences. Even though we do not resolve the general Shift Equivalence  Problem, we obtain a positive answer for the class of imprimitivity bimodules.
\end{abstract}

%%%%%%%%%%%%%%%%
\section{Introduction}
%%%%%%%%%%%%%%%%

The concept of a $\ca$-correspondence provides a unified framework for the study of a variety of diverse objects, including the crossed product of $\ca$-algebras by $\bbZ$ and the Cuntz-Krieger algebras. Since their introduction, with the seminal paper of Pimsner \cite{Pim97}, the $\ca$-correspondences and their associated $\ca$-algebras have been studied by many authors, including Abadie, Eilers and Exel \cite{AEE98}, Deaconu, Kumjian, Pask and Sims \cite{DKPS10}, Doplicher, Pinzari and Zuccante \cite{DPZ98}, Fowler, Muhly and Raeburn \cite{FMR03}, Fowler and Raeburn \cite{FowRae99}, Kajiwara, Pinzari and Watatani \cite{KPW98}, Katsura \cite{Kat03, Kat04}, Kwa\'{s}niewski \cite{Kwa09}, Muhly, Pask and Tomforde \cite{MPT08}, Muhly and Solel \cite{MuhSol98, MuhSol00}. Traditionally, the language of $\ca$-correspondences has been used as a mean to generalize concrete results either from the theory of Cuntz-Krieger $\ca$-algebras or crossed product $\ca$-algebras to a broader context. (See \cite{Kat03} for an exposition). Our approach here is similar in spirit to that of \cite{KakKat11} and one of our goal is to explore this path in both ways. In other words, not only we extend constructions from graphs and non-negative integral matrices to the language of $\ca$-correspondences but in addition we use our abstract theory to obtain concrete results, which are new even for the special case of graph $\ca$ algebras or other operator algebras, e.g, Corollary \ref{dyn sys}.

In \cite{Wil74} Williams introduced three relations for the class of matrices with non-negative integer entries, which are successively weaker. Two such matrices $E$ and $F$ are said to be \emph{elementary strong shift equivalent} (symb. $E\sse F$) if there exist matrices $R$ and $S$ with non-negative integer entries such that $E = RS$ and $F=SR$. The transitive closure of the elementary strong shift equivalence is called \textit{strong shift equivalence} (symb. $\tsse$), and $E$ is \emph{shift equivalent} to $F$ (symb. $E \se F$) if there exist matrices $R,S$ with non-negative integer entries such that $E^n = RS$, $F^n=SR$ and $ER=SF$, $FR=SE$ for some $n\in \bbN$.

The main goal of Williams was to characterize the topological conjugacy of subshifts of finite type using algebraic criteria. In \cite{Wil74}, he proved that two subshifts $\sigma_E$ and $\sigma_F$ of finite type are topologically conjugate if and only if $E \tsse F$.  Williams also claimed that the relations $\tsse$ and $\se$ are equivalent, thus providing a more manageable criterion for the conjugacy of shifts. Unfortunately, an error in \cite{Wil74} made invalid the proof of that last assertion, and the equivalence of $\tsse$ and $\se$ remained an open problem for over than 20 years, known as \emph{Williams' Conjecture}. The breakthrough came with the work of Kim and Roush \cite{KimRou99} who proved that the Williams Conjecture for the class of non-negative integral matrices is false. Their work reshaped the Williams' Conjecture into what is known today as the \emph{Shift Equivalence Problem}, which for a particular class $\S$ of matrices with entries in a certain ring $R$  asks whether $\tsse$ is equivalent to $\se$ \textit{within} $\S$.

Williams' notions of shift equivalence carry over to the class of C*- correspondences if one replaces in the above definitions the matrices $E$ and $F$ with C*-correspondences and the multiplication of matrices with the internal tensor product. (See Section \ref{S:rel}  for the precise definitions.) This introduces three notions of a relation between $\ca$-correspondences, which will be denoted again as $\sse ,\tsse$ and $\se$. There exists also a fourth equivalence relation, named \textit{strong Morita equivalence} (symb. $\sme$), which generalizes the concept of unitary conjugacy for matrices to the realm of $\ca$-correspondences.

The concept of strong Morita equivalence for $\ca$-correspondences was first developed and studied by Abadie, Eilers and Exel \cite{AEE98}, and Muhly and Solel \cite{MuhSol00}. Among other results these authors show that if $E \sme F$ then the associated Cuntz-Pimsner algebras $\O_E$ and $\O_F$ are (strong) Morita equivalent as well.  The notion of elementary and strong shift equivalence for $\ca$-correspondences was first studied by Muhly, Pask and Tomforde \cite{MPT08}. These authors also prove that strong shift equivalence of $\ca$-correspondences implies the Morita equivalence of the associated Cuntz-Pimsner algebras, thus extending classical results of Cuntz and Krieger \cite{CunKri80}, Bates \cite{Bat02} and Drinen and Sieben \cite{DriSie01} for graph $\ca$-algebras. In their study of strong shift equivalence Muhly, Pask and Tomforde \cite{MPT08} raise two conjectures, which turn out to be important for the further development of the theory \cite[Remark 5.5]{MPT08}:

\begin{quote}
\noindent \textbf{Conjecture 1.} Let $E$ and $F$ be two non-degenerate, regular $\ca$-correspondences and let $E_\infty$ and $F_\infty$ be their associated Pimsner dilations. If $E \tsse F$, then $E_\infty \sme F_\infty$.

\noindent \textbf{Conjecture 2.} Let $E$ and $F$ be two non-degenerate, regular $\ca$-correspondences and let $E_\infty$ and $F_\infty$ be their associated Pimsner dilations. If $E_\infty \sme F_\infty$, then $E \tsse F$.
\end{quote}

The concept of shift equivalence has been studied extensively from both the dynamical and the ring theoretic viewpoint. (See \cite{Wag99} for  a comprehensive exposition.) In general, shift equivalence  has been  recognized to be a more manageable  invariant than strong shift equivalence, as it is decidable over certain rings \cite{KimRou88}. Unlike  strong shift equivalence, the study of shift equivalence, from the viewpoint of $\ca$-correspondences, has  been met with limited success \cite{Mat07}. (Other operator theoretic viewpoints however have been quite successful \cite{Kri80}.)

There are three major objectives that are being met in this work. First we complete the study of strong shift equivalence of Muhly, Tomforde and Pask \cite{MPT08} by settling both of their conjectures: with the extra  requirement of fullness, Conjecture 1 is settled in the affirmative (see Theorem \ref{T:full} and the remarks preceding it), while Conjecture 2 has a negative answer (Theorem \ref{T:se not inv} and Remark \ref{R:conj 2}).

A second objective is the detailed study of the shift equivalence for $\ca$-correspondences. First, we raise the analogues of Conjectures 1 and 2 for shift equivalence (instead of strong shift equivalence) and we discover that the answers are the same as in the case of strong shift equivalence. Using that information we prove  Theorem \ref{T:se gives me}, which states that if two non-degenerate, regular, full C*-correspondences $E$ and $F$ are shift equivalent, then their corresponding Cuntz-Pimsner algebras $\O_E$ and $\O_F$ are (strong) Morita equivalent. This generalizes results of Cuntz and Krieger \cite{CunKri80}, and Muhly, Pask and Tomforde \cite{MPT08} and appears to be new even for Cuntz-Krieger algebras, where the two notions of shift and strong shift equivalence are known to be different (Corollary \ref{graph}). Combined with the work of Williams our result says that if two subshifts of finite type are eventually conjugate \cite{Wag99}, then their Cuntz-Krieger algebras are strong Morita equivalent (Corollary \ref{dyn sys}).

Our third goal in this paper is the introduction of the Shift Equivalence Problem in the context of $\ca$-correspondences. In light of our previous discussion, it seems natural to ask whether strong shift equivalence and shift equivalence coincide for $\ca$-correspondences. We coin this problem as the Shift Equivalence Problem  for $\ca$-correspondences. The work of Kim and Roush \cite{KimRou99} shows that the Shift Equivalence Problem  has a negative answer within the class of graph correspondences, but it leaves open the option for a positive answer within the whole class of $\ca$-correspondences. In general, we do not know the answer even though the work of Kim and Roush hints that it should be negative. In spite of this, we show that the Shift Equivalence Problem has a positive answer for imprimitivity bimodules: all four notions of ``equivalence'' described in this paper coincide for imprimitivity bimodules, Theorem \ref{T:sse=se=sme}.

There are more things accomplished in this paper, and we describe each of them within the appropriate sections. Most notably, we settle a third conjecture of Muhly, Pask and Tomforde coming from \cite{MPT08} by showing that \cite[Theorem 3.14]{MPT08} is valid without the assumption of non-degeneracy (Theorem \ref{T:main 6}). That is, regular (perhaps degenerate) strong shift equivalent $\ca$-correspondences have strong Morita equivalent Cuntz-Pimsner algebras.

Finally we discuss on the importance of Conjectures 1 and 2 regarding the behavior of the shift relations when dilating. Pimsner dilations are often the natural objects arising in the $\ca$-literature for invertible relations, and these objects have been under thorough investigation. For example the Pimsner dilation of an injective $\ca$-dynamical system is an automorphic dynamical system \cite{Sta93}, thus the Cuntz-Pimsner algebra is a usual crossed product. (Under this prism our work here can be used in combination to provide further invariants for the Shift Equivalence Problem.) From this point of view one may wonder whether there is a \textit{general} feature of the Pimsner dilation that makes it unique in a certain sense. In Theorem~\ref{T:minimal dilation} we answer this by showing that the Pimsner dilation $X_\infty$ of $X$ is the unique essential Hilbert bimodule that shares the same Cuntz-Pimsner algebra with $X$. Moreover it is the minimal essential Hilbert bimodule containing $X$ (Theorem~\ref{T:min sub}). Further applications of this result are of independent interest and are to be pursued elsewhere.

The paper is organized as follows. In Section \ref{S:preliminaries} we establish the notation and terminology to be used throughout this paper. In Section \ref{S:shifts} we explore the concept of a dilation for a $\ca$-correspondence. In addition we prove the minimality of the Pimsner dilation. Sections \ref{S:rel}, \ref{S:equiv} and \ref{S:app} contain the main results of this paper.

%%%%%%%%%%%%%%%%%%%%%%%%%%%%%%%%
\section{Preliminaries}\label{S:preliminaries}
%%%%%%%%%%%%%%%%%%%%%%%%%%%%%%%%

We use Lance's book \cite{Lan95} as a general reference for Hilbert $\ca$-modules and Katsura's paper \cite{Kat04} for $\ca$-correspondences.

An \emph{inner-product right $A$-module} over a $\ca$-algebra $A$ is a linear space $X$ which is a right $A$-module together with an $A$-valued inner product. For $\xi \in X$ we define $\nor{\xi}_X:= \nor{\sca{\xi,\xi}_A}_A^{1/2}$. The $A$-module $X$ will be called a \emph{right Hilbert $A$-module} if it is complete with respect to the norm $\nor{\cdot}_X$. In this case $X$ will be denoted by $X_A$. It is straightforward to prove that if $(a_i)$ is an approximate unit in $A$ or in the closed ideal $\sca{X,X}_X$, then $(a_i)$ is also a right contractive approximate unit (c.a.i.) for $X$.

Dually we call $X$ a \emph{left Hilbert $A$-module} if it is complete with respect to the norm induced by a \emph{left $A$-module inner-product } $\lsca{\cdot,\cdot}_X$. The term \emph{Hilbert module} is reserved for the right Hilbert modules, whereas the left case will be clearly stated.

Given a Hilbert $A$-module $X$ over $A$, let $X^*= \{\xi^* \in \L(X,A)\mid \xi^*(\zeta)= \sca{\xi,\zeta}_X \}$ be the \emph{dual left Hilbert $A$-module}, with
\begin{align*}
a\cdot \xi^*= (\xi a^*)^* \text{ and }
\lsca{\xi^*,\zeta^*}_{X^*}= \sca{\xi,\zeta}_X,
\end{align*}
for all $\xi, \zeta\in X$ and $a\in A$.

For $X, Y$ Hilbert $A$-modules let $\L(X,Y)$ be the (closed) linear space of the adjointable maps. For $\xi\in X$ and $y\in Y$, let $\Theta_{y,\xi}\in \L(X,Y)$  such that $\Theta_{y,\xi}(\xi')= y \sca{\xi,\xi'}_X$, for all $\xi' \in X$. We denote by $\K(X,Y)$ the closed linear subspace of $\L(X,Y)$ spanned by $\{\Theta_{y,\xi}: \xi\in X, y \in Y\}$. If $X=Y$ then $\K(X,X)\equiv \K(X)$ is a closed ideal of the $\ca$-algebra $\L(X,X)\equiv \L(X)$.

%%%%%%%%%%%%%%%%%%%%%%%%%%%%%%%%
\begin{lemma}\label{L:cai for compact}
Let $X, Y, Z$ be Hilbert $A$-modules. If $\sca{X,X}_X$ provides a right c.a.i. $(a_i)$ for $Y$, then $\overline{\K(X,Y)\K(Z,X)}= \K(Z,Y)$.
\end{lemma}
\begin{proof}
The existence of the right c.a.i. $(a_i)$ implies that $Y\sca{X,X}$ is dense in $Y$, hence $\overline{\K(X,Y) \K(Z,X)} = \overline{\K(Z,Y\sca{X,X})}= \K(Z,Y)$.
\end{proof}

%%%%%%%%%%%%%%%%%%%%%%%%%%%%%%%%
\begin{definition}
An \emph{$A$-$B$-correspondence $X$} is a right Hilbert $B$-module together with a $*$-homomorphism $\phi_X\colon A \rightarrow \L(X)$. We denote this by ${}_A X_B$. When $A=B$ we refer to $X$ as a \emph{$\ca$-correspondence over $A$}.

A subspace $Y$ of $X$ is a \emph{subcorrespondence} of ${}_A X_B$, if it is a $C$-$D$-correspondence for some $\ca$-subalgebras $C$ and $D$ of $A$ and $B$, respectively.
\end{definition}

An $A$-$B$-correspondence $X$ is called \emph{non-degenerate} (resp. \textit{strict}) if the closed linear span of $\phi_X(A)X$ is equal to $X$ (resp. complemented in $X$). We say that $X$ is \emph{full} if $\sca{X,X}_X$ is dense in $B$. Finally, $X$ is called \emph{regular} if both it is \emph{injective}, i.e., $\phi_X$ is injective, and $\phi_X(A) \subseteq \K(X)$.

Two $A$-$B$-correspondences $X$ and $Y$ are called unitarily equivalent (symb. $X \approx Y$) if there is a unitary $u \in \L(X,Y)$ such that $u(\phi_X(a)\xi b)= \phi_Y(a)(u \xi )b$, for all $a\in A, b\in B, \xi \in X$.

%%%%%%%%%%%%%%%%%%%%%%%%%%%%%%%%
\begin{example}
Every Hilbert $A$-module $X$ becomes a regular $\K(X)$-$A$- correspondence when endowed with the left multiplication $\phi_X\equiv \id_{\K(X)}: \K(X) \rightarrow \L(X)$. A left inner product over $\K(X)$ can be defined by $\lsca{\xi, \eta}_{X}= \Theta_{\xi, \eta}$, for all $\xi, \eta \in X$. Also $X^*$ is an $A$-$\K(X)$-correspondence, when endowed with the following operations
\begin{align*}
\sca{\xi^*, \eta^*}_{X^*}=\lsca{\xi,\eta}_X, \, \xi^* \cdot k =(k^*\xi)^*, \textup{ and }
\phi_{X^*}(a)\xi^*= a\cdot \xi^*= (\xi \cdot a^*)^*,
\end{align*}
for all $\xi, \eta \in X$, $k\in \K(X)$ and $a\in A$.
\end{example}

%%%%%%%%%%%%%%%%%%%%%%%%%%%%%%%%
\begin{example}\label{E:K(X,Y)}
For Hilbert $A$-modules $X$ and $Y$, the space $\L(X,Y)$ becomes an $\L(Y)$-$\L(X)$-correspondence by defining $\sca{s,t}:=s^*t$, $t\cdot a:= ta$ and $b\cdot t:=bt$, for every $s,t\in \L(X,Y), a\in \L(X)$ and $b\in \L(Y)$.

Trivially, $\K(X,Y)$ is a $\K(Y)$-$\K(X)$-subcorrespondence of $\L(X,Y)$. Note that, when $\sca{X,X}_X$ provides a right c.a.i. for $Y$, then $\K(Y)$ acts faithfully on $\K(X,Y)$. When $X=Y$ this is automatically true.
\end{example}

For two $\ca$-correspondences ${}_A X_B$ and ${}_B Y_C$, the \emph{interior} or \emph{stabilized tensor product}, denoted by $X \otimes_B Y$ (or simply by $X \otimes Y$) is the quotient of the vector space tensor product $X \otimes_{\alg} Y$ by the subspace generated by elements of the form
\begin{align*}
\xi b \otimes y - \xi \otimes \phi(b)y, \foral \xi \in X, y \in Y, b
\in B.
\end{align*}
It becomes a Hilbert $C$-module when equipped with
\begin{align*}
& (\xi \otimes y)c:= \xi \otimes (y c), & (\xi \in X, y\in Y, c
\in C),\\
& \sca{\xi_1\otimes y_1, \xi_2\otimes y_2}_{X\otimes Y}:= \sca{y_1, \phi(\sca{\xi_1,\xi_2}_X)y_2}_Y, & (\xi_1, \xi_2 \in X, y_1, y_2 \in
Y).
\end{align*}
For $s\in \L(X)$ let $s\otimes \id_Y \in \L(X\otimes Y)$ be the map $\xi\otimes y \mapsto (s\xi)\otimes y$. Then $X \otimes_B Y$ becomes an $A$-$C$-correspondence by defining $\phi_{X\otimes Y}(a):= \phi_X(a) \otimes \id_Y$. See \cite[Chapter 4]{Lan95} for further details.

The interior tensor product plays the role of a generalized associative multiplication of $\ca$-correspondences and the following lemmas will be useful in the sequel.

%%%%%%%%%%%%%%%%%%%%%%%%%%%%%%%%
\begin{lemma}\label{L:cai ten pr}
Let ${}_A X_B$ and ${}_B Y_C$ be $\ca$-correspondences. If $(c_i)$ is an approximate identity of $\sca{Y,Y}_Y$, then $(c_i)$ is a right c.a.i. for the interior tensor product $X \otimes_B Y$.
\end{lemma}
\begin{proof}
The norm on $X\otimes_B Y$ is a submultiplicative tensor norm, i.e., $\nor{\xi \otimes y} \leq \nor{\xi}\nor{y}$ for all $\xi \in X, y\in Y$. Thus $\lim_i y c_i = y$ implies $\lim_i (\xi\otimes y)c_i = \lim_i \xi \otimes (y c_i) = \xi \otimes y$, for a c.a.i. $(c_i)$ as above.
\end{proof}

%%%%%%%%%%%%%%%%%%%%%%%%%%%%%%%%
\begin{lemma}\label{L:sme non-deg}
Let ${}_A X_B$ and ${}_B Y_C$ be two $\ca$-correspondences. If ${}_A X_B$ is non-degenerate then $X\otimes_B Y$ is non-degenerate.
\end{lemma}
\begin{proof}
Immediate since the tensor norm is submultiplicative.
\end{proof}

%%%%%%%%%%%%%%%%%%%%%%%%%%%%%%%%
\begin{lemma}\label{L:compact ten pr}
Let $X, Y$ be Hilbert $A$-modules and ${}_A Z_B$ be an injective $\ca$- correspondence. Then the mapping
\begin{align*}
\otimes\id_Z\colon \L(X,Y) \rightarrow
\L(X\otimes_A Z, Y \otimes_A Z): t\mapsto t\otimes \id_Z
\end{align*}
is isometric. In particular, if ${}_A Z_B$ is regular then $\K(X,Y)$ is mapped inside $\K(X \otimes_A Z, Y \otimes_A Z)$.
\end{lemma}
\begin{proof}
Let the mappings
\begin{align*}
& \Psi \equiv \otimes \id_Z\colon \L(X,Y) \rightarrow \L(X\otimes_A Z, Y\otimes_A Z),\\
& \psi \equiv \otimes \id_Z \colon \L(X) \rightarrow \L(X\otimes_A Z),
\end{align*}
and note that $\psi$ is a $*$-homomorphism. If $t \otimes \id_Z =0$ for $t\in \L(X)$, then we obtain
\[
0 = \sca{t\xi \otimes \zeta, t\xi \otimes \eta}_{X \otimes Z} = \sca{\zeta,\phi_Z(\sca{t\xi,t\xi}_X) \eta}_{X \otimes Z},
\]
for all $\zeta, \eta \in Z$ and $\xi \in X$. Thus $\phi_Z(\sca{t\xi, t\xi}_X) =0$, and injectivity of $\phi_Z$ implies that $t=0$. Thus $\psi$ is injective, hence isometric. Since
\[
\Psi(t_1)^* \Psi(t_2)= \psi(t_1^* t_2), \quad \mbox{for all } t_1, t_2 \in \L(X,Y),
\]
we obtain that $\Psi$ is also isometric.

Finally, assume that $Z$ is regular and let $(s_i)$ be a right approximate unit for $\K(X,Y)$ inside $\K(X)$. Then $k\otimes \id_Z =  \lim_i (k\otimes \id_Z) \cdot \psi(s_i)$ for $k\in \K(X,Y)$, by the previous paragraph. Nevertheless \cite[Proposition 4.7]{Lan95} shows that $\psi(s_i)\in \K(X\otimes_A Z)$ and the conclusion now follows by noting that $(k\otimes \id_Z) \cdot \K(X\otimes_A Z) \subseteq \K(X\otimes_A Z, Y\otimes_A Z)$.
\end{proof}

%%%%%%%%%%%%%%%%%%%%%%%%%%%%%%%%
\begin{example}\label{Ex:x^* otimes x}
When a Hilbert $A$-module $X$ is considered as the $\K(X)$-$A$-correspondence, then $X \otimes_A X^* \approx \K(X)$ as $\ca$-correspondences over $\K(X)$, via the mapping $u_1\colon \xi \otimes \zeta^* \mapsto \Theta_{\xi,\zeta}$, and $X^* \otimes_{\K(X)} X \approx \sca{X,X}_X$, as $\ca$-correspondences over $A$, via the mapping $u_2\colon \xi^*\otimes \zeta \mapsto \sca{\xi,\zeta}$. In particular $X^*\otimes_{\K(X)} X \approx A$, when $X$ is full.
\end{example}

%%%%%%%%%%%%%%%%%%%%%%%%%%%%%%%%
\begin{definition}
A \emph{Hilbert $A$-$B$-bimodule} is a $\ca$-correspondence ${}_A X_B$ together with a left inner product $\lsca{\cdot,\cdot}_X\colon X \times X \rightarrow A$, which satisfies
\begin{align*}
&\lsca{\xi,\eta}_X \cdot \zeta=\xi \cdot \sca{\eta,\zeta}_X, & (\xi, \eta,\zeta \in X).
\end{align*}
An injective Hilbert bimodule will be called \emph{essential}.

An \emph{$A$-$B$-imprimitivity bimodule or equivalence bimodule} is an $A$-$B$- bimodule $M$ which is simultaneously a full left and a full right Hilbert $A$-module, i.e., $\lsca{M,M}_M=A$ and $\sca{M,M}_M=B$.
\end{definition}

There are several equivalent conditions for a $\ca$-correspondence ${}_A X_A$ to be a Hilbert bimodule, e.g., see \cite{Kak13, Kat03}. We will often use that ${}_A X_A$ is a Hilbert bimodule if and only if the restriction of $\phi_X$ to Katsura's ideal \cite{Kat04}
\begin{align*}
J_X=\ker\phi_X^\bot \cap \phi_X^{-1}(\K(X))
\end{align*}
is onto $\K(X)$. Therefore all Hilbert bimodules ${}_A X_A$ are non-degenerate. Moreover ${}_A X_A$ is essential if and only if $J_X$ is an essential ideal of $A$.

An imprimitivity bimodule ${}_A M_B$ is automatically non-degenerate and regular because $A\simeq^{\phi_M} \K(M)$. It is immediate that $M$ is an $A$-$B$-imprimitivity bimodule if and only if $M^*$ is a $B$-$A$-imprimitivity bimodule and Example \ref{Ex:x^* otimes x} induces the following.

%%%%%%%%%%%%%%%%%%%%%%%%%%%%%%%%
\begin{lemma}\label{L:imprimitivity}
If $M$ is an $A$-$B$-imprimitivity bimodule, then $M \otimes_B M^* \approx A$ and $M^* \otimes_A M \approx B$, where $A$ and $B$ are the trivial $\ca$-correspondences over themselves.
\end{lemma}

There is a number of ways of considering \emph{a} direct sum of Hilbert modules and  \emph{a} direct sum of $\ca$-correspondences. For example for ${}_A E_A$ and ${}_B S_A$ one may consider the (\emph{interior}) direct sum Hilbert $A$-module $\begin{bmatrix} E \\ S \end{bmatrix}$ which becomes an $(A \oplus B)$-$A$-correspondence in the obvious way. Moreover, given ${}_A R_B$, one may consider the (\emph{exterior}) direct sum Hilbert $A \oplus B$-module $\begin{bmatrix} E & R \end{bmatrix}$ which becomes an $A$-$(A \oplus B)$-correspondence in the obvious way. These constructions are subcorrespondences of a matrix, that we give below.

Given ${}_A E_A$, ${}_A R_B$, ${}_B S_A$ and ${}_B F_B$, the \emph{matrix $\ca$-correspondence} $X= \left[\begin{array}{c|c} E & R\\ S & F \end{array}\right]$ over $A\oplus B$ is the Hilbert $(A\oplus B)$-module of the linear space of the ``matrices'' $\left[\begin{array}{c|c} e & r \\ s & f \end{array}\right]$, $e\in E, r\in R, s\in S, f\in F$, with
{\small
\begin{align*}
\left[\begin{array}{c|c} e & r \\ s & f \end{array}\right] \cdot (a,b)
& := \left[\begin{array}{c|c} ea & rb \\ sa & fb \end{array}\right],\\
\sca{\left[\begin{array}{c|c} e_1 & r_1 \\ s_1 & f_1 \end{array}\right], \left[\begin{array}{c|c} e_2 & r_2 \\ s_2 & f_2 \end{array}\right]}_X
& := \big(\sca{e_1,e_2}_E + \sca{s_1,s_2}_S, \sca{r_1,r_2}_R + \sca{f_1,f_2}_F\big),
\end{align*}}
and the $*$-homomorphism $\phi\colon A\oplus B \rightarrow \L\left(\left[\begin{array}{c|c} E & R\\ S & F \end{array}\right]\right)$ is defined by the ``diagonal'' action
\begin{align*}
\phi(a,b) \left[\begin{array}{c|c} e & r \\ s & f \end{array}\right]
:= \left[\begin{array}{c|c} \phi_E(a)e & \phi_R(a)r \\ \phi_S(b)s & \phi_F(b)f \end{array}\right].
\end{align*}
Then $E, R, S$ and $F$ imbed naturally as subcorrespondences in $\left[\begin{array}{c|c} E & R\\ S & F \end{array}\right]$, since the latter is exactly the exterior direct sum $\ca$-correspondence of the two interior direct sum $\ca$-correspondences $\left[\begin{array}{c} E\\ S \end{array} \right]$ and $\left[\begin{array}{c} R\\ F \end{array} \right]$.\\

The following lemma explains the terminology ``matrix $\ca$-correspondence'', as tensoring corresponds to ``matrix multiplication''.

%%%%%%%%%%%%%%%%%%%%%%%%%%%%%%%%%%%%%%%%%%%%%%%%%%%%%
\begin{lemma}\label{L:technical}
Let $E,F,R,S$ (resp. $E',F',R',S'$) be $\ca$-correspondences as above. Then
\begin{align*}
&\left[\begin{array}{c|c} E & R\\ S & F \end{array}\right]
\otimes_{A\oplus B} \left[\begin{array}{c|c} E' & R'\\ S' & F'
\end{array}\right] \approx
\left[\begin{array}{c|c}
\left[\begin{array}{c} E\otimes_A E'\\ R\otimes_B S' \end{array}\right] &
\left[\begin{array}{c} E\otimes_A R' \\ R \otimes_B F' \end{array}\right] \\ & \\
\left[\begin{array}{c} S\otimes_A E' \\ F\otimes_B S' \end{array}\right] &
\left[\begin{array}{c} S\otimes_A R' \\ F\otimes_B F' \end{array}\right]
\end{array}\right].
\end{align*}
\end{lemma}
\begin{proof}
Note that all entries in the second matrix make sense. It is a matter of routine calculations to show that the mapping
\begin{align*}
\left[\begin{array}{c|c} e_1 & r_1 \\ s_1 & f_1
\end{array}\right] \otimes \left[\begin{array}{c|c} e_2 &
r_2 \\ s_2 & f_2 \end{array}\right] \mapsto
\left[\begin{array}{c|c}
\left[\begin{array}{c} e_1\otimes e_2\\ r_1\otimes s_2 \end{array}\right] &
\left[\begin{array}{c} e_1 \otimes r_2 \\ r_1 \otimes f_2 \end{array}\right] \\ & \\
\left[\begin{array}{c} s_1 \otimes e_2 \\ f_1\otimes s_2 \end{array}\right] &
\left[\begin{array}{c} s_1\otimes r_2 \\ f_1 \otimes f_2 \end{array}\right]
\end{array}\right].
\end{align*}
defines the unitary element that gives the equivalence.
\end{proof}

It will be convenient to omit the zero entries, when possible. For example, we write $\K(E,S)$ instead of $\K\left( \left[\begin{array}{c|c} E & 0 \\ 0 & 0 \end{array}\right], \left[\begin{array}{c|c} 0 & 0 \\ S & 0 \end{array}\right]\right)$.\\

A (\emph{Toeplitz}) \emph{representation} of ${}_A X_A$ into a $\ca$-algebra $B$, is a pair $(\pi,t)$, where $\pi\colon A \rightarrow B$ is a $*$-homomorphism and $t\colon X \rightarrow B$ is  a linear map, such that $\pi(a)t(\xi)=t(\phi_X(a)(\xi))$ and $t(\xi)^*t(\eta)=\pi(\sca{\xi,\eta}_X)$, for $a\in A$ and $\xi,\eta\in X$. An application of the $\ca$-identity shows that $t(\xi)\pi(a)=t(\xi a)$ is also valid. A representation $(\pi , t)$ is said to be \textit{injective} if $\pi$ is injective; in that case $t$ is an isometry.

The $\ca$-algebra generated by a representation $(\pi,t)$ equals the closed linear span of $t^n(\bar{\xi})t^m(\bar{\eta})^*$, where for simplicity $\bar{\xi} \equiv \xi_1 \otimes \cdots \otimes \xi_n \in X^{\otimes n}$ and $t^n(\bar{\xi})\equiv t(\xi_1)\dots t(\xi_n)$. For any representation $(\pi,t)$ there exists a $*$-homomorphism $\psi_t:\K(X)\rightarrow B$, such that $\psi_t(\Theta^X_{\xi,\eta})= t(\xi)t(\eta)^*$ \cite[Lemma 2.2]{KPW98}.

Let $J$ be an ideal in $\phi_X^{-1}(\K(X))$; we say that a representation $(\pi,t)$ is $J$-coisometric if $\psi_t(\phi_X(a))=\pi(a), \text{ for any } a\in J$.
Following \cite{Kat04}, the $J_{X}$-coisometric representations $(\pi,t)$, for
\begin{align*}
J_X=\ker\phi_X^\bot \cap \phi_X^{-1}(\K(X)),
\end{align*}
are called \emph{covariant representations}.

The \emph{Toeplitz-Cuntz-Pimsner algebra} $\T_X$ is the universal $\ca$-algebra for ``all'' representations of $X$, and the \emph{Cuntz-Pimsner algebra} $\O_X$ is the universal $\ca$-algebra for all covariant ($J_X$-coisometric) representations of $X$. The universal $\ca$-algebra $\O(J,X)$ for ``all'' $J$-covariant representations of $X$ is called the \emph{$J$-relative Cuntz-Pimsner algebra} \cite{FMR03, MuhSol98}. \emph The \emph{tensor algebra} $\T_{X}^+$ is the norm-closed algebra generated by the universal copy of $A$ and $X$ in $\T_X$ \cite{MuhSol98}.

Even though there are various manifestations of the Cuntz-Pimsner algebras appearing before Katsura's work \cite{Kat04}, in this paper we adhere strictly to the notation and terminology of \cite{Kat04}. By \cite[Lemma 1.2]{KakPet13}, Katsura's ideal is the biggest ideal for obtaining an isometric copy of $A$, hence of $X$ in $\O_X$, and so  Katsura's ideal is the ``correct'' ideal to describe the minimal covariant $\ca$-algebra that contains isometric copies of the original construction ${}_A X_A$. (Compare also with \cite[Proposition 7.14]{Kat07} and the subsequent comments.) Relative Cuntz-Pimsner algebras beyond that point contain a quotient of the original $\ca$-correspondence $X$ over $A$. Nevertheless, for the majority of the paper, the $\ca$-correspondences to be considered will be injective and therefore Katsura's notion of the Cuntz-Pimsner algebra, the one adopted by the present authors, will coincide with other notions by earlier authors.

A powerful tool for identifying faithful representations of both $\T_X$ and $\O_X$ are the various gauge-invariance uniqueness theorems. For the Cuntz-Pimsner-Toeplitz algebra of a regular $\ca$-correspondence, such a result was given by Fowler and Raeburn \cite[Theorem 2.1]{FowRae99}. Variations of the gauge-invariance uniqueness theorems for Cuntz-Pimsner algebras were given by Doplicher, Pinzari and Zuccante \cite[Theorem 3.3]{DPZ98}, extended later by Fowler, Muhly and Raeburn \cite[Theorem 4.1]{FMR03} to injective $\ca$-correspondences. The reader should notice that injectivity is a blanket assumption in \cite{FMR03} even though it is not explicitly stated \cite[Remarks at the beginning of Section 1]{FMR03}. Therefore \cite[Theorem 4.1]{FMR03} cannot generalize gauge-invariance uniqueness theorems for general $\ca$-correspondences, e.g., $\ca$-correspondences associated to graphs with sinks. In full generality gauge-invariance uniqueness theorems for $\T_X$ and $\O_X$ were finally given by Katsura \cite[Theorem 6.2, Theorem 6.4]{Kat04}. An alternative proof for $\O_X$ was also given afterwards by Muhly and Tomforde \cite{MuhTom04}.

If $X$ is an $A$-$B$-correspondence, then we may identify $X$ with the $(A\oplus B)$-correspondence $\left[\begin{array}{c|c} 0 & X \\ 0 & 0 \end{array}\right]$ and thus define the Toeplitz-Cuntz-Pimsner, the Cuntz-Pimsner and the tensor algebra of $X$ as the corresponding algebras of the $(A\oplus B)$-correspondence $\left[\begin{array}{c|c} 0 & X \\ 0 & 0 \end{array}\right]$.

If $X$ is a $\ca$-correspondence, then $X^*$ may not be a $\ca$-correspondence (in the usual right-side sense) but it can be described as follows: if $(\pi_u,t_u)$ is the universal representation of ${}_A X_A$, then $X^*$ is the closed linear span of $t_u(\xi)^*, \xi \in X$ with the left multiplication and inner product inherited by the correspondence $\ca(\pi_u,t_u)$. Nevertheless, $X^*$ is an imprimitivity bimodule, whenever $X$ is.

%%%%%%%%%%%%%%%%%%%%%%%%%%%%%%%%
\begin{example}\label{Ex:imp}
Let $X_\G$ be the $\ca$-correspondence coming from a graph $\G$ (see \cite{Rae05}). Then $X_\G$ is an imprimitivity bimodule if and only if $\G$ is either a cycle or a double-infinite path. In that case, $X_{\G}^*$ is the correspondence coming from the graph having the same edges as $\G$ but its arrows reversed.

If $\ga\colon A \rightarrow B$ is a $*$-homomorphism, then the associated correspondence $B_\ga$ is defined by
\begin{align*}
& \sca{\xi_1,\xi_2} = \xi_1^*\xi_2, &(\xi_1, \xi_2 \in B)\\
& a\cdot \xi \cdot b = \ga(a)\xi b, & (a \in A, \xi \in B, b\in B).
\end{align*}
Then $B_\ga$ is an imprimitivity bimodule if and only if $\ga$ is a $*$-isomorphism. In that case, $B_{\ga}^* = A_{\ga^{-1}}$.
\end{example}

%%%%%%%%%%%%%%%%%%%%%%%%%%%%%%%%
\section{Dilations of $\ca$-correspondences}\label{S:shifts}
%%%%%%%%%%%%%%%%%%%%%%%%%%%%%%%%

For $\ca$-correspondences ${}_A X_A$ and ${}_B Y_B$ we write $X \lesssim Y$ when $X$ is unitarily equivalent to a subcorrespondence of $Y$. In other words there is a $*$-injective representation $\pi\colon A \rightarrow B$, a  $\pi(A)$-$\pi(A)$-~subcorrespondence $Y_0$ of $Y$ and a unitary $u\in \L(X, Y_0)$ such that $u:X \rightarrow Y_0$, where $X$ is (now) considered a $\ca$-correspondence over $\pi(A)$.

Here we do not impose $u$ to be an isometry in $\L(X,Y)$, i.e., that $u$ has a complemented range in $Y$ \cite[Theorem 3.2]{Lan95}. What we ask is that $u$ is an isometric map and $X \approx_{(\pi,u)} Y_0 \subseteq Y$, in terms of the representation theory of $\ca$-correspondences.

%%%%%%%%%%%%%%%%%%%%%%%%%%%%%%%%
\begin{definition}
A $\ca$-correspondence ${}_B Y_B$ is said to be \emph{a dilation} of the correspondence ${}_A X_A$, if ${}_A X_A \lesssim {}_B Y_B$ and the associated Cuntz-Pimsner algebras $\O_X$ and $\O_Y$ are $*$-isomorphic.
\end{definition}

%%%%%%%%%%%%%%%%%%%%%%%%%%%%%%%%
\subsection{Pimsner Dilations of $\ca$-correspondences}
%%%%%%%%%%%%%%%%%%%%%%%%%%%%%%%%

Given an injective $\ca$-correspondence $X$ there is a natural way to pass to an injective Hilbert bimodule $X_\infty $, such that $X_\infty$ is a dilation of $X$. This construction was first introduced by Pimsner in \cite{Pim97}. In \cite[Appendix A]{KakKat11} we revisited such a construction by using direct limits. As it appears both constructions produce the same essential Hilbert bimodule, something that is not proved in \cite[Appendix A]{KakKat11}. Below we give a brief description (slightly corrected) and the appropriate identification. As we will show, Pimsner dilation is uniquely characterized by a minimal condition, analogously to minimal \emph{maximal} representations of non-selfadjoint operator algebras.

For the description of Pimsner's dilation \cite{Pim97} we use the modern terminology. Fix an injective covariant pair $(\pi,t)$ of $X$ that admits a gauge action, and let
\[
A_\infty:= \O_X^\be= E(\O_X) \equiv \pi(A) + \overline{\{ \psi_{t^n}(\K(X^{\otimes n})) \mid n\geq 1\}}.
\]
Define $X_\infty := \overline{t(X) \cdot A_\infty}$, where ``$\cdot$'' denotes the usual multiplication of operators. Then $X_\infty$ is an $A_\infty$-subcorrespondence of $\ca(\pi,t)$ and $J_{X_\infty} = \overline{\{ \psi_{t^n}(\K(X^{\otimes n})) \mid n\geq 1\}}$. An approximate identity in $\psi_t(\K(X))$ is an approximate identity of $J_{X_\infty}$, and one can show that $J_{X_\infty}$ is essential in $A$. Even more $X_\infty$ is an (essential) Hilbert bimodule with $\O_{X_\infty} \simeq \O_X$. This is mainly \cite[Theorem 2.5]{Pim97}, when using the identification $X_\infty \approx X \otimes_A A_\infty$. Note that $A_\infty$ and $X_\infty$ do not depend on the choice of $(\pi,t)$, because of the gauge-invariance uniqueness theorem \cite[Theorem 6.4]{Kat04}.

We give a brief description of the dilation as in \cite[Appendix A]{KakKat11}. Let the isometric mapping $\tau\colon X \rightarrow \L(X,X^{\otimes 2})$, such that $\tau_\xi(\eta)=\xi \otimes \eta$, and consider the direct limits
{\footnotesize
\begin{align*}
X
\stackrel{\tau}{\longrightarrow}
\L(X, X^{\otimes 2})
\stackrel{\otimes{\id_X}}{\longrightarrow} \dots
&\longrightarrow \varinjlim (\L(X^{\otimes n}, X^{\otimes n+1}), \otimes\id_X)=:Y\\
A
\stackrel{\phi_X}{\longrightarrow}
\L(X) \stackrel{\otimes{\id_X}}{\longrightarrow}
\L(X^{\otimes 2})
\stackrel{\otimes{\id_X}}{\longrightarrow} \dots
&\longrightarrow \varinjlim ( \L(X^{\otimes n}), \otimes\id_X)=:B.
\end{align*}
}By Lemma \ref{L:compact ten pr} the connecting maps are isometric. If $r\in \L(X^{\otimes n})$, $s\in \L(X^{\otimes n},X^{\otimes n+1})$ and $[r], [s]$ are their equivalence classes in $B$ and $Y$ respectively, then we define $[s]\cdot [r]:= [sr]$. Injectivity of the connecting maps implies that this action is well-defined and can be extended to a right $B$-action on $Y$. Similarly, we may define a $B$-valued right inner product on $Y$ by setting
\begin{align*}
\sca{[s'],[s]}_Y \equiv [(s')^*s] \in B.
\end{align*}
for $s, s' \in \L(X^{\otimes n},X^{\otimes n+1})$, $n \in \bbN$, and then extending to $Y \times Y$. Finally, we define a $*$-homomorphism $\phi_Y \colon B \rightarrow \L(Y)$ by setting
\begin{align*}
\phantom{XXX} \phi_Y([r])([s]) \equiv [rs], \quad r\in \L(X^{\otimes
n}), s\in \L(X^{\otimes n-1},X^{\otimes n}), n\geq0,
\end{align*}
and extending to all of $B$ by continuity. We therefore have a left $B$-action on $Y$ and thus $Y$ becomes a $\ca$-correspondence over $B$.

Let $Y_0$ be the closed subspace of $Y$  that is generated by the copies of $\K(X^{\otimes n}, X^{\otimes n+1})$, for $n\in \bbZ_+$. Moreover let $B_0$ be the $\ca$-subalgebra of $B$ that is generated by the copies of $\K(X^{\otimes n})$, for $n\in \bbZ_+$. Then $Y_0$ is a $B_0$-subcorrespondence of $Y$, that contains ${}_A X_A$. In particular, when $\phi_X(A) \subseteq \K(X)$ then $\tau(X) \subseteq \K(X,X^{\otimes 2})$ and Lemma \ref{L:compact ten pr} implies that
\[
Y_0 = \varinjlim (\K(X^{\otimes n}, X^{\otimes n+1}), \gs_n)
\text{ and }
B_0 = \varinjlim (\K(X^{\otimes n}), \rho_n),
\]
where $\gs_0 = \tau$, $\rho_0=\phi_X$ and $\gs_n = \rho_n = \otimes \id_X$ for $n >1$.

Let us sketch a proof for the unitary equivalence of $Y_0$ with $X_\infty$. Assume first that ${}_A X_A$ is regular, so $J_X =A$. Let $(\pi,t)$ be the covariant Fock representation in $\L(\F(X))/ \K(\F(X))$. For $k\in \K(X^{\otimes n})$ observe that
\[
\psi_{t^n}(k) = 0 \oplus \cdots \oplus k \oplus (k \otimes \id_X) \oplus \cdots + \K(\F(X)).
\]
Then the family $\{\psi_{t^n}\}$ is compatible with the directed system since
\[
\psi_{t^{n+1}}(k\otimes \id_X) - \psi_{t^n}(k) = 0 \oplus \cdots \oplus k \oplus 0 \oplus \cdots \,\, \in \, \K(\F(X)),
\]
thus it defines a representation of the direct limit $\varinjlim (\K(X^{\otimes n}), \rho_n)$ into $A_\infty$. It is straightforward that it is injective and onto. Hence it induces a unitary equivalence of $Y_0 = \varinjlim (\K(X^{\otimes n}, X^{\otimes n+1}), \gs_n)$ with $X_\infty$, since $Y_0$ is the closure of $\tau(X) \cdot \varinjlim (\K(X^{\otimes n}), \rho_n)$ in $Y$, and $X_\infty = \overline{t(X) \cdot A_\infty}$.

In order to show that $Y_0 \approx X_\infty$ for the general case we have to prove that the family $\{\psi_{t^n}\}$ induces a $*$-isomorphism between $B_0$ and $A_\infty$. What is required is to prove that $\{\psi_{t^n}\}$ is compatible with the equivalence classes in the direct limit; hence to show that if $k\in \K(X^{\otimes n})$ with $k\otimes \id_X \in \K(X^{\otimes n+1})$ then $k\in \K(X^{\otimes n} J_X)$.  This is derived by a similar argument as in the proof of \cite[Proposition 5.9]{Kat04} and \cite[Lemma 2.6]{FMR03}.

%%%%%%%%%%%%%%%%%%%%%%%%%%%%%%%%
\begin{theorem} \label{T:inj}
\cite[Theorem 2.5]{Pim97}, \cite[Theorem 6.6]{KakKat11} Let $X$ be an injective $\ca$-correspondence and let $X_{\infty}$  be the $A_\infty$-correspondence constructed above. Then $X_{\infty}$ is an essential Hilbert bimodule and the Cuntz-Pimsner algebras $\O_{X_\infty}$ and $\O_X$ coincide.
\end{theorem}

In the sequel, the $A_\infty$-correspondence $X_{\infty}$ appearing in the theorem above will be called the \textit{Pimsner dilation of $X$}.

%%%%%%%%%%%%%%%%%%%%%%%%%%%%%%%%
\begin{remark}
If ${}_A X_A$ is a Hilbert bimodule then it coincides with its Pimsner dilation. This follows from \cite[Proposition 5.18]{Kat04}, or Theorem \ref{T:min sub} that follows.
\end{remark}

%%%%%%%%%%%%%%%%%%%%%%%%%%%%%%%%
\begin{remark}\label{R:not imp}
When $X$ is non-degenerate and full, then its Pimsner dilation $X_\infty$ is also full, hence an imprimitivity bimodule. Indeed,  it suffices to prove that
\[
\overline{\K(X^{\otimes n+1}, X^{\otimes n}) \K(X^{\otimes n}, X^{\otimes n+1})}= \K(X^{\otimes n}), \]
for all $n \geq 1$. By Lemma \ref{L:cai for compact}, it suffices to show that $\sca{X^{\otimes n+1}, X^{\otimes n+1}}$ provides a c.a.i. for $X^{\otimes n}$. For $n=2$,
\begin{align*}
\sca{X^{\otimes 2}, X^{\otimes 2}} = \sca{X, \phi_X(\sca{X,X})X} = \sca{X, \phi_X(A)X} =\sca{X,X}=A,
\end{align*}
and an inductive argument completes the claim.

However, when $X$ is not full then this is not true (even when $X$ is non-degenerate). For example, let $X$ be the non-degenerate, injective correspondence coming from the infinite tail graph
\begin{align*}
\xymatrix{ \bullet^{v_0} & \bullet^{v_1}
\ar@/_/[l]_{e_0} & \bullet^{v_2}
\ar@/_/[l]_{e_1} & \cdots \ar@/_/[l]_{e_2}
}
\end{align*}
which is not full, since $\sca{X,X}= A \ominus \bbC \de_{v_0}$. On the other hand $X$ is a Hilbert bimodule since $\phi_X(\de_{v_n}) = \Theta_{\de_{e_n}, \de_{e_n}}$, and therefore coincides with its Pimsner dilation.
\end{remark}

%%%%%%%%%%%%%%%%%%%%%%%%%%%%%%%%
\begin{example}\label{Ex:dil dyn}
Let $(A,\ga)$ denote a dynamical system where $\ga$ is a $*$-injective endomorphism of $A$. We can define the direct limit dynamical system $(A_\infty,\ga_\infty)$ by
\begin{align*}
\xymatrix{
A \ar[r]^{\ga} \ar[d]^\ga &
A \ar[r]^{\ga} \ar[d]^\ga &
A \ar[r]^{\ga} \ar[d]^\ga &
\cdots \ar[r] &
A_\infty \ar[d]^{\ga_\infty} \\
A \ar[r]^\ga &
A \ar[r]^\ga &
A \ar[r]^\ga &
\cdots \ar[r] &
A_\infty
}
\end{align*}
The limit map $\ga_\infty$ is an automorphism of $A_\infty$ and extends $\ga$ (note that $A$ imbeds in $A_\infty$ since $\ga$ is injective). Then the $A_\infty$-$A_\infty$-correspondence $X_{\ga_\infty}$, is the Pimsner dilation of $X_\ga$. Moreover, $X_{\ga_\infty}$ satisfies a minimal condition, i.e., $A_\infty = \overline{\bigcup_{n \in \bbZ} \ga^{-n}(A)}$ \cite{Sta93}.
\end{example}

In analogy, we introduce a ``minimality condition'' that ${}_A X_A$ has as a subcorrespondence of an injective ${}_B Y_B$. This will be that $X\cdot B$ is dense in $Y$. Note that this assumption excludes trivial extensions. Indeed, if there was a submodule $Z$ of $Y$ such that $Z\perp X$, then $Z \perp (X\cdot B)$, therefore $Z \perp Y$, hence $Z=\{0\}$.

We also note that under these assumptions for ${}_A X_A$ and ${}_B Y_B$ we may view $\K(X)$ as a $\ca$-subalgebra of $\K(Y)$. Indeed, a compact operator $k = \lim_N \sum_{i=1}^N \Theta^X_{\xi_i,\eta_i}$ in $\K(X)$ extends to the compact operator of the same form $k'= \lim_N \sum_{i=1}^N \Theta^Y_{\xi_i,\eta_i}$ in $\K(Y)$. This extension is unique since it necessarily satisfies $k'(\xi \cdot b) = k(\xi)b$. Thus in the sequel we will use the same symbol for $k \in \K(X)$ and its extension in $\K(Y)$.

%%%%%%%%%%%%%%%%%%%%%%%%%%%%%%%%
\begin{lemma}\label{L:cov}
Let ${}_A X_A$ be an injective subcorrespondence of an injective correspondence ${}_B Y_B$ such that $Y=\overline{X \cdot B}$. Then $J_X  \subseteq J_Y$. Consequently, every covariant pair $(\pi,t)$ of ${}_B Y_B$ defines a covariant pair $(\pi|_A,t|_X)$ for ${}_A X_A$, therefore $\O_X \subseteq \O_Y$.
\end{lemma}
\begin{proof}
Since $X$ and $Y$ are injective we obtain that $J_X = \phi^{-1}_X(\K(X))$ and $J_Y = \phi^{-1}_Y(\K(Y))$. Let $a\in J_X$. Then $\phi_X(a)=k \in \K(X)$ and we can view $k$ as an operator in $\K(Y)$ of the same form. Then for every $\xi \in X$ and $b\in B$ we obtain that
\begin{align*}
\phi_Y(a)(\xi \cdot b)
 = \phi_Y(a)(\xi) b = \phi_X(a)(\xi)b
 = k(\xi)b = k(\xi \cdot b).
\end{align*}
By assumption $X\cdot B$ is dense in $Y$, hence $\phi_Y(a)(y)=k(y)$ for all $y\in Y$. Therefore $\phi_Y(a)=k \in \K(Y)$, which implies that $a\in J_Y$.

For the second part of the proof, let $(\pi,t)$ be a covariant pair for $Y$ and note that $\psi_{t|_X}(k) = \psi_t(k)$ for every $k\in \K(X) \subseteq \K(Y)$. Indeed, for $\xi, \eta \in X$,
\begin{align*}
\psi_{t|_X}(\Theta^X_{\xi,\eta}) = t|_X(\xi)t|_X(\eta)^* = t(\xi)t(\eta)^* = \psi_t(\Theta^Y_{\xi,\eta}),
\end{align*}
since $Y = \overline{X \cdot B}$. Thus for $a\in J_X$ we get
\begin{align*}
\pi|_A(a)=\pi(a) = \psi_t(\phi_Y(a))=\psi_t(k) = \psi_{t|_X}(k) = \psi_{t|_X}(\phi_X(a))
\end{align*}
therefore $(\pi|_A,t|_X)$ is a covariant pair for $X$. Moreover, a gauge action $\{\be_z\}$ for $(\pi,t)$ is inherited to $(\pi|_A,t|_X)$ and the proof is complete.
\end{proof}

%%%%%%%%%%%%%%%%%%%%%%%%%%%%%%%%
\begin{theorem}\label{T:min sub}
The Pimsner dilation $X_\infty$ of an injective correspondence $X$ is a minimal Hilbert bimodule that contains $X$ as a subcorrespondence.
\end{theorem}
\begin{proof}
Let ${}_B Y_B$ be a Hilbert bimodule such that $X \preceq Y \preceq X_\infty$. Up to unitary equivalence, we can assume that $X\subseteq Y \subseteq X_\infty$ as subcorrespondences, hence $A \subseteq B \subseteq A_\infty$ as subalgebras. In particular $\phi_\infty|_Y=\phi_Y$ and $\K(X) \subseteq \K(Y) \subseteq \K(X_\infty)$. It suffices to prove that $B=A_\infty$, because then
\[
X_\infty = \overline{X \cdot A_\infty} \subseteq \overline{Y \cdot A_\infty} = \overline{Y \cdot B} = Y \subseteq X_\infty,
\]
hence $Y=X_\infty$, where the symbol ``$\cdot$'' denotes the right action for $X_\infty$.

\smallskip

\noindent \emph{Claim.} The ideal $J_Y$ of $B$ is contained in the ideal $J_\infty$ of $A_\infty$.

\noindent \emph{Proof of Claim.} Recall that $X_\infty$ is injective, thus $J_\infty = \phi^{-1}_\infty(\K(X_\infty))$. Let $b\in J_Y = \phi_Y^{-1}(\K(Y)) \cap \ker\phi_Y^\perp$; in particular there is a $k_Y \in \K(Y)$ such that $\phi_Y(b)=k_Y$. The compact operator $k_Y$ can be also viewed as a compact operator $k_\infty$ in $\K(X_\infty)$ (of the same form) and it suffices to show that $\phi_\infty(b)=k_\infty$. To this end let an element $\eta \in X_\infty$; by the form of $X_\infty$ we can assume that $\eta=\xi \cdot a$, for some $\xi \in X$ and $a\in A_\infty$. Then
\begin{align*}
\phi_\infty(b)(\xi \cdot a)
& = \phi_\infty(b)(\xi)\cdot a \\
& = \phi_\infty(b)|_Y(\xi) \cdot a \\
& = \phi_Y(b)(\xi)\cdot a \\
& = k_Y(\xi)\cdot a \\
& = k_\infty(\xi)\cdot a \\
& = k_\infty(\xi \cdot a),
\end{align*}
where we used the fact that $\xi \in Y$. This proves the claim.

\smallskip

Fix an injective covariant pair $(\pi,t)$ of $X_\infty$ that admits a gauge action. Since $J_X, J_Y \subseteq J_\infty$, then $\O_X =\ca(\pi|_A,t|_X)$, $\O_Y=\ca(\pi|_B,t|_Y)$ and $\O_{X_\infty}=\ca(\pi,t)$. Therefore
\begin{align*}
\pi(A) \subseteq \pi(B) \subseteq \pi(A_\infty) = \pi(A) + \overline{\Span} \{ \psi_{(t|_X)^n}(\K(X^{\otimes n})) \mid n\geq 1 \}.
\end{align*}
Showing that $\psi_{(t|_X)^n}(\K(X^{\otimes n})) \subseteq \pi(B)$, for every $n \geq 1$, will complete the proof. By assumption $Y$ is a Hilbert bimodule, therefore $\psi_{t|_Y}(\K(Y^{\otimes n})) \subseteq \pi(B)$ for all $n\geq 1$. But then
\begin{align*}
\psi_{(t|_{X})^n} (\K(X^{\otimes n}))
& = \psi_{t^n} (\K(X^{\otimes n})) \\
& \subseteq \psi_{t^n} (\K(Y^{\otimes n}))\\
& = \psi_{(t|_Y)^n} (\K(Y^{\otimes n})) \\
& \subseteq \pi(B),
\end{align*}
for every $n\geq 1$, which completes the proof.
\end{proof}

%%%%%%%%%%%%%%%%%%%%%%%%%%%%%%%%
\begin{theorem}\label{T:minimal dilation}
Let ${}_A X_A$ be an injective correspondence. Then the Pimsner dilation is contained in any correspondence ${}_B Y_B$ that satisfies:
\begin{enumerate}
\item $X$ is a subcorrespondence of $Y$,
\item $Y$ is an essential Hilbert bimodule,
\item $Y=\overline{X \cdot B}$,
\end{enumerate}
where ``$\cdot$'' denotes the right action in $Y$ that extends the right action of $A$.
\end{theorem}
\begin{proof}
As proved, $X_\infty$ is such a correspondence. In order to show that $X_\infty$ is a subcorrespondence of a $Y$ that satisfies the above, it suffices to prove that $A_\infty \subseteq B$. Indeed, then
\[
X_\infty  = \overline{X \cdot A_\infty} \subseteq \overline{X \cdot B} = Y,
\]
where ``$\cdot$'' denotes the right action in $Y$.

Fix an injective covariant pair $(\pi,t)$ of $Y$ that admits a gauge action. Then $(\pi|_A, t|_X)$ is an injective covariant pair for $X$ that admits a gauge action by Lemma \ref{L:cov}, and
\[
A_\infty \simeq \pi(A) + \overline{\Span}\{ \psi_{(t|_X)^n} (\K(X^{\otimes n}))\mid n \geq 1\}.
\]
Since $Y$ is a Hilbert bimodule, then $\phi_Y(J_Y)=\K(Y)$. Thus
\begin{align*}
\psi_{t|_X}(\K(X))
 = \psi_t(\K(X)) \subseteq \psi_t(\K(Y))
 = \psi_t(\phi_Y(J_Y))=\pi(J_Y) \subseteq \pi(B).
\end{align*}
Moreover,
\begin{align*}
\psi_{(t|_X)^2}(\K(X^{\otimes 2}))
& \subseteq  \overline{t(X)\psi_{t|_X}(\K(X)) t(X)^*} \\
& \subseteq \overline{t(X) \pi(B) t(X)^*}\\
& = \overline{t(X \cdot B) t(X)^*} \\
& \subseteq \overline{t(Y)t(Y)^*} \\
& = \psi_t(\K(Y)) \\
& \subseteq \pi(B).
\end{align*}
Inductively, we get that $\psi_{(t|_X)^n}(\K(X^{\otimes n})) \subseteq \pi(B)$, for all $n \geq 1$. Trivially $\pi(A) \subseteq \pi(B)$ since $X$ is a subcorrespondence of $Y$, therefore $A_\infty$ is ($*$-isomorphic to) a subalgebra of $\pi(B)$, hence $A_\infty \subseteq B$.
\end{proof}

Moreover the Pimsner dilation of an injective $\ca$-correspondence is the unique minimal dilation of $X$.

%%%%%%%%%%%%%%%%%%%%%%%%%%%%%%%%
\begin{theorem}\label{T:main 7}
Let ${}_A X_A$ be an injective subcorrespondence of some ${}_B Y_B$. Then $Y \approx X_\infty$ if and only if
\begin{enumerate}
\item $Y$ is an essential Hilbert bimodule,
\item $Y=\overline{X \cdot B}$,
\item $\O_Y = \O_X$.
\end{enumerate}
Then $Y_B$ is minimal as a Hilbert bimodule containing $X_A$, automatically.
\end{theorem}
\begin{proof}
As in the proof of Theorem \ref{T:main 6} it suffices to show that $B\simeq A_\infty$. To this end fix an injective covariant pair $(\pi,t)$ of $Y$ that admits a gauge action. Then $(\pi|_A,t|_X)$ is an injective covariant pair of $X$ that inherits the gauge action of $(\pi,t)$ by restriction, and we can assume that $\ca(\pi|_A,t_X)= \O_X =\O_Y= \ca(\pi,t)$. Therefore $A_\infty$ is $*$-isomorphic to the fixed point algebra $\O_X^{\be}$. Since $\O_X= \O_Y$ and $Y_B$ is a Hilbert bimodule we obtain
\[
A_\infty \simeq \O_X^\be = E(\O_X) = E(\O_Y) = \pi(B).
\]
Therefore $A_\infty \simeq B$ and the proof is complete.
\end{proof}

%%%%%%%%%%%%%%%%%%%%%%%%%%%%%%%%
\subsection{Dilations for Compacts}
%%%%%%%%%%%%%%%%%%%%%%%%%%%%%%%%

The construction of the Pimsner dilation applies to more general settings. Indeed, let $X, Y$ be Hilbert $A$-modules and let ${}_A Z_B$ be a regular correspondence. By Lemma \ref{L:compact ten pr}, one can form the following directed systems {\footnotesize
\begin{align*}
\K(X,Y) \stackrel{\otimes{\id_Z}}{\longrightarrow} \K(X\otimes Z,
Y\otimes Z) \stackrel{\otimes{\id_Z}}{\longrightarrow} \dots
&\longrightarrow \varinjlim (\K(X \otimes Z^{\otimes n}, Y
\otimes Z^{\otimes n}), \otimes\id_Z),
\\
\K(X) \stackrel{\otimes{\id_Z}}{\longrightarrow} \K(X\otimes Z)
\stackrel{\otimes{\id_Z}}{\longrightarrow} \dots &\longrightarrow \varinjlim
( \K(X \otimes Z^{\otimes n}), \otimes\id_Z),
\\
\K(Y) \stackrel{\otimes{\id_Z}}{\longrightarrow} \K(Y\otimes Z)
\stackrel{\otimes{\id_Z}}{\longrightarrow} \dots &\longrightarrow \varinjlim
( \K(Y \otimes Z^{\otimes n}), \otimes\id_Z).
\end{align*}
}For simplicity, we will write $(\K(X,Y),\id_Z)_\infty$ for the direct
limit
\begin{align*}
\varinjlim \left( \K(X \otimes Z^{\otimes n}, Y \otimes Z^{\otimes
n}), \id_Z\right).
\end{align*}
By imitating the proofs in \cite[Appendix A]{KakKat11} and Lemma \ref{L:compact ten pr}, we see that $(\K(X,Y), \id_Z)_\infty$ is a $\K(Y)_\infty$-$\K(X)_\infty$-correspondence, which will be called \emph{the dilation of $\K(X,Y)$ by $Z$}, or more simply the \emph{$Z$-dilation of $\K(X,Y)$}. If in addition $\sca{X\otimes Z^{\otimes n},X\otimes Z^{\otimes n}}$ provides a right c.a.i. for  $Y\otimes Z^{\otimes n}$, for every $n\geq 0$, then by Example \ref{E:K(X,Y)} and  Lemma \ref{L:cai for compact} $(\K(X,Y), \id_Z)_\infty$ is a regular $\ca$-correspondence. In the case of a regular  $\ca$-correspondence ${}_A X_A$, the Pimsner dilation $X_\infty$ is simply the $X$-dilation of $\K(X,X^{\otimes 2})$, which is always regular.

%%%%%%%%%%%%%%%%%%%%%%%%%%%%%%%%
\begin{proposition}\label{P:wz=zw}
Let $X,Y$ be Hilbert $A$-modules and $Z, W$ be regular $\ca$-correspondences over $A$. If $Z \otimes_A W \approx W \otimes_A Z$, then
\begin{align*}
(\K(X,Y), \id_Z)_\infty \lesssim (\K(X\otimes_A W, Y\otimes_A W),
\id_Z)_\infty.
\end{align*}
\end{proposition}
\begin{proof}
We will identify $Z \otimes W$ with $W \otimes Z$. Since $Z$ commutes with $W$ the diagram{\footnotesize
\begin{align*}
\xymatrix{ \K(X,Y) \ar^{\id_{\K(X,Y)}}[d] \ar^{\id_{\K(X,Y)}}[r] & \K(X,Y) \ar^{\otimes{\id_W} }[d] \ar^{{\otimes{\id_Z} } \ }[r] & \K(X \otimes Z,
Y \otimes Z) \ar^{\otimes{\id_W} }[d] \ar^{\qquad \otimes{\id_Z} }[r] &
\dots \\
\K(X,Y) \ar^{\otimes{\id_W} \qquad }[r] &\K(X\otimes W, Y\otimes W)  \ar^{\otimes{\id_Z} \qquad}[r] &
\K(X \otimes W \otimes Z, Y \otimes W \otimes Z)  \ar^{\qquad \qquad \qquad \otimes{\id_Z} }[r] &
\dots }
\end{align*}
}is commutative and defines a linear map
\begin{align*}
s\colon (\K(X,Y), \id_Z)_\infty \rightarrow (\K(X\otimes_A W, Y\otimes_A
W), \id_Z)_\infty.
\end{align*}
That is, if $[k]\in (\K(X,Y),\id_Z)_\infty$, such that $k \in \K(X \otimes Z^{\otimes n}, Y \otimes Z^{\otimes n})$ then
{\small
\begin{align*}
k\otimes \id_W \in \K(X\otimes Z^{\otimes n} \otimes W, Y\otimes Z^{\otimes n} \otimes W) = \K(X\otimes  W \otimes  Z^{\otimes n}, Y\otimes Z^{\otimes n}),
\end{align*}
}and we define $s[k]= [k \otimes \id_W]$. Note that $s$ is defined on $(\K(X),\id_Z)_\infty$ and on $(\K(Y),\id_Z)_\infty$ in the analogous way. For example we get the commutative diagram
{\footnotesize
\begin{align*}
\xymatrix{ \K(X) \ar^{\id_{\K(X)}}[d] \ar^{\id_{\K(X)}}[r] & \K(X)  \ar^{\otimes{\id_W} }[d] \ar^{\otimes{\id_Z} \ }[r] & \K(X \otimes Z) \ar^{\otimes{\id_W} }[d] \ar^{\qquad \otimes{\id_Z} }[r] &
\dots \\
\K(X) \ar^{\otimes{\id_W} }[r] &\K(X\otimes W)  \ar^{\otimes{\id_Z} \quad}[r] &
\K(X \otimes W \otimes Z)  \ar^{\qquad \quad \otimes{\id_Z} }[r] &
\dots }
\end{align*}}
It is a matter of routine computations to show that $s$ is a $(\K(Y),\id_Z)_\infty$-$(\K(X),\id_Z)_\infty$-mapping. What is left to show is that $s$ is also a unitary onto its range.

It suffices to prove this at the $n$-th level. For $k_1, k_2 \in \K(X\otimes Z^{\otimes n},Y\otimes Z^{\otimes n})$, then
\begin{align*}
\sca{s[k_1], s[k_2]}
& = \sca{[k_1 \otimes \id_{ W}], [k_2\otimes \id_{W}]} \\
& = [k_1 \otimes \id_{W}]^* [k_2 \otimes \id_{W}] \\
& = [(k_1^*\circ k_2)\otimes \id_{W}] \\
& = [(k_1^*\circ k_2)] \\
& = \sca{[k_1], [k_2]},
\end{align*}
where we have identified $(\K(X), \id_Z)_\infty$ with its image via $s$.
\end{proof}

For the proof of the following Proposition, recall that there is a major difference between isometric mappings of $\ca$-correspondences and mappings of $\ca$-correspondences that are isometries (of Hilbert modules). However, when an isometric mapping is also onto then it is a unitary between Hilbert modules.

%%%%%%%%%%%%%%%%%%%%%%%%%%%%%%%%
\begin{proposition}\label{P:tens reg shifts}
Let $X, Y, W$ be Hilbert $A$-modules and $Z$ be a regular $\ca$-correspondence over $A$. If the ideal $\sca{X\otimes Z^{\otimes n},X\otimes Z^{\otimes n}}$ of $A$ provides a right c.a.i. for  $Y\otimes Z^{\otimes n}$, for all $n \geq 0$, then
\begin{align*}
(\K(X,Y), \id_Z)_\infty \otimes_{(\K(X),\id_Z)_\infty} (\K(W,X),
\id_Z)_\infty \approx (\K(W, Y), \id_Z)_\infty.
\end{align*}
\end{proposition}
\begin{proof}
In order to ``visualize'' the proof, imagine that we multiply ``vertically'' and term by term the correspondences
\begin{align*}
\K(X,Y) \stackrel{\otimes{\id_Z} }{\longrightarrow} \K(X\otimes_A Z,
Y\otimes_A Z) \stackrel{\otimes{\id_Z} }{\longrightarrow} \dots
&\longrightarrow(\K(X, Y), \id_Z)_\infty\\
\K(W,X) \stackrel{\otimes{\id_Z} }{\longrightarrow} \K(W\otimes_A Z,
X\otimes_A Z) \stackrel{\otimes{\id_Z} }{\longrightarrow} \dots
&\longrightarrow (\K(W, X), \id_Z)_\infty
\end{align*}
in the order $\K(X\otimes_A Z^{\otimes n}, Y\otimes_A Z^{\otimes n}) \cdot \K(W\otimes_A Z^{\otimes n}, X\otimes_A Z^{\otimes n})$.

We consider $\K(X\otimes Z^{\otimes n})$, $\K(Y\otimes Z^{\otimes n})$ and $\K(W\otimes Z^{\otimes n})$ as $\ca$-subalgebras of $(\K(X),\id_Z)_\infty$, $(\K(Y),\id_Z)_\infty$ and $(\K(W),\id_Z)_\infty$, respectively, for every $n\in \bbZ_+$. The proof is divided into two parts.

For the first part of the proof note that for every $n \in \bbZ_+$, the Banach space $\K(X\otimes_A Z^{\otimes n}, Y\otimes_A Z^{\otimes n})$ (resp. $\K(W\otimes_A Z^{\otimes n}, X\otimes_A Z^{\otimes n})$) is an injective $\K(Y\otimes_A Z^{\otimes n})$-$\K(X\otimes_A Z^{\otimes n})$-correspondence (resp. a $\K(X\otimes_A Z^{\otimes n})$-$\K(W\otimes_A Z^{\otimes n})$-correspondence) in the obvious way.
Set  {\small
\begin{align*}
\fA_n:= \K(X\otimes_A Z^{\otimes n}, Y\otimes_A Z^{\otimes n})  \otimes_{\K(X\otimes_A Z^{\otimes n})}   \K(W\otimes_A Z^{\otimes n}, X\otimes_A Z^{\otimes n}).
\end{align*}
}For every $n \in \bbZ_+$, the mapping
\begin{align*}
\rho_n\colon \fA_n \rightarrow \fA_{n+1}: k \otimes t \mapsto (k\otimes \id_Z)\otimes (t\otimes \id_Z),
\end{align*}
is isometric on sums of elementary tensors and therefore extends to the closures. Thus $\fA_n \lesssim \fA_{n+1}$, and we can form the direct limit $\ca$-correspondence
\begin{align*}
\fA_0 \stackrel{\rho_0 }{\longrightarrow} \fA_1 \stackrel{\rho_1 }{\longrightarrow} \fA_2 \stackrel{\rho_2 }{\longrightarrow} \dots \longrightarrow \varinjlim(\fA_n,\rho_n),
\end{align*}
which is a $(\K(Y),\id_Z)_\infty$-$(\K(W),\id_Z)_\infty$-correspondence. Now the mappings
{\small
\begin{align*}
\phi_n\colon
\fA_n \rightarrow
(\K(X,Y),\id_Z)_\infty \otimes_{(\K(X),\id_Z)_\infty}  (\K(W,X),\id_Z)_\infty : k\otimes t \mapsto [k]\otimes [t]
\end{align*}
}are compatible with the directed system and define a $\ca$-correspondence mapping
\begin{align*}
\phi\colon \varinjlim(\fA_n,\rho_n) \rightarrow
(\K(X,Y),\id_Z)_\infty \otimes_{(\K(X),\id_Z)_\infty}  (\K(W,X),\id_Z)_\infty,
\end{align*}
which is isometric since every $\phi_{n+1}$ is. Thus
\begin{align*}
\varinjlim(\fA_n,\rho_n) \lesssim(\K(X,Y),\id_Z)_\infty \otimes_{(\K(X),\id_Z)_\infty}  (\K(W,X),\id_Z)_\infty,
\end{align*}
via $\phi$. However $(\K(X,Y),\id_Z)_\infty \otimes_{(\K(X),\id_Z)_\infty}  (\K(W,X),\id_Z)_\infty$ is spanned by the elements $[k]\otimes[t]$ for $k\otimes t\in \fA_n$. Therefore $\phi$ is onto and so
\begin{align*}
\varinjlim(\fA_n,\rho_n) \approx
(\K(X,Y),\id_Z)_\infty \otimes_{(\K(X),\id_Z)_\infty}  (\K(W,X),\id_Z)_\infty
\end{align*}

For the second part of the proof we construct an isometric map from $\varinjlim(\fA_n,\rho_n)$ onto $(\K(W,Y),\id_Z)_\infty$. Start by  defining the maps
\begin{align*}
u_n\colon \fA_n \rightarrow (\K(W, Y), \id_Z)_\infty:
k\otimes t \mapsto [kt].
\end{align*}
These are well defined since $kt\in \K(W\otimes Z^{\otimes n}, Y\otimes Z^{\otimes n})$ and
\begin{align*}
u_n((k\cdot a) \otimes t)= [(k\cdot a)t]= [k(a\cdot t)]= u_n(k\otimes (a \cdot t)),
\end{align*}
due to the associativity of the multiplication. Also, the maps $u_n$ are isometric and compatible with the direct sequence. Therefore the family $\{u_n\}$ defines an isometric map $u$ from the direct limit into $(\K(W,Y),\id_Z)_\infty$, which extends to a mapping of $\ca$-correspondences. Thus $\fA_n \lesssim (\K(W,Y),\id_Z)_\infty$ and so $\varinjlim(\fA_n,\rho_n) \lesssim (\K(W,Y),\id_Z)_\infty$. Finally note that the assumption that the ideals $\sca{X\otimes Z^{\otimes n}, X \otimes Z^{\otimes n}}$ provide a c.a.i. for each $Y\otimes Z^{\otimes n}$, combined with Lemma \ref{L:cai for compact}, implies that the isometric maps $u_n$ are onto and hence $u$ is onto.
\end{proof}

%%%%%%%%%%%%%%%%%%%%%%%%%%%%%%%%
\section{Relations associated with $\ca$-correspondences}\label{S:rel}
%%%%%%%%%%%%%%%%%%%%%%%%%%%%%%%%

In \cite{MuhSol00} Muhly and Solel introduced the notion of Morita equivalence of $\ca$-correspondences. This concept generalizes the notion of outer conjugacy for $\ca$-dynamical systems \cite[Proposition 2.4]{MuhSol00}.

%%%%%%%%%%%%%%%%%%%%%%%%%%%%%%%%
\begin{definition}\label{D:sme}
The $\ca$-correspondences ${}_A E_A$ and ${}_B F_B$ are called \emph{strong Morita equivalent} if there is an imprimitivity bimodule ${}_A M_B$ such that $E \otimes_A M \approx M \otimes_B F$. In that case we write $E \sme F$.
\end{definition}

Muhly and Solel \cite{MuhSol00} examined this relation under the assumption that the $\ca$-correspondences are both non-degenerate and injective. Nevertheless, non-degeneracy is automatically implied. Indeed, if $E \sme F$ via $M$, then
\begin{align*}
E
\approx E \otimes_A A
\approx E \otimes_A M \otimes_B M^*
\approx M\otimes_B F \otimes_B M^*
\approx M\otimes_B (F \otimes_B M^*),
\end{align*}
and Lemma \ref{L:sme non-deg} implies that $E$ is non-degenerate. A symmetrical argument applies for $F$.

%%%%%%%%%%%%%%%%%%%%%%%%%%%%%%%%
\begin{remark}
In contrast to non-degeneracy, injectivity is not automatically implied by $\sme$. For example, pick your favorite non-degenerate and non-injective $\ca$-correspondence ${}_A E_A$ and let $E=F$. Then $E\sme E$ via the trivial imprimitivity bimodule ${}_A A_A$, but both $E$ and $F$ are not injective.
\end{remark}

In \cite{MPT08} Muhly, Pask and Tomforde introduced the notion of elementary strong shift equivalence between $\ca$-correspondences that generalizes the corresponding notion for graphs.

%%%%%%%%%%%%%%%%%%%%%%%%%%%%%%%%
\begin{definition}\label{D:esse}
Let ${}_A E_A$ and ${}_B F_B$ be $\ca$-correspondences. Then $E$ and $F$ will be called \emph{elementary strong shift equivalent} (symb.  $E\sse F$) if there are $\ca$-correspondences ${}_A R_B$ and ${}_B S_A$ such that $E\approx R\otimes_B S$ and $F\approx S\otimes_A R$ as $\ca$-correspondences.
\end{definition}

Elementary strong shift equivalence is obviously symmetric. Moreover it is reflexive for non-degenerate $\ca$-correspondences; indeed, $E \sse E$ via $E$ and $A$. However, it is unclear whether or not it is transitive. By \cite[Example 2]{Wil74} we get that $\sse$ is not transitive when restricted to the class of non-negative integral matrices, but this doesn't mean that $\sse$ is not transitive for the whole class of C*-correspondences.

Nevertheless, we have the following proposition.

%%%%%%%%%%%%%%%%%%%%%%%%%%%%%%%%
\begin{proposition}\label{P:condition trans}
Let ${}_A E_A, {}_B F_B$ and ${}_C G_C$ be $\ca$-correspondences. Assume that $E \sse F$ via $R, S$ and $F\sse G$ via $T, Z$. If either $Z$ or $R$ is an imprimitivity bimodule, then $E \sse G$.
\end{proposition}
\begin{proof}
For $E, F$ and $G$ as above we have that
\begin{align*}
E\approx R \otimes_B S \, , F\approx S\otimes_A R \, , F\approx T \otimes_C Z \, , G\approx Z \otimes_B T.
\end{align*}
Assume that $Z$ is an imprimitivity bimodule (a symmetric argument can be used if $R$ is an imprimitivity bimodule). Then by Lemma \ref{L:imprimitivity} $Z^*\otimes_C Z \approx B$ and $Z \otimes_B Z^* \approx C$. Hence,
\begin{align*}
(Z \otimes_B S) \otimes_A (R \otimes_B Z^*)
& \approx Z \otimes_B (S\otimes_A R) \otimes_B Z^* \\
& \approx Z \otimes_B F \otimes_B Z^* \\
& \approx Z \otimes_B T \otimes_C Z \otimes_B Z^* \\
& \approx Z \otimes_B T \otimes_C C \\
& \approx Z\otimes_B (T\cdot C) \\
& = Z \otimes_B T \\
& \approx G.
\end{align*}
On the other hand,
\begin{align*}
(R \otimes_B Z^*) \otimes_C (Z \otimes_B S)
 \approx R \otimes_B B  \otimes_B S
 = (R \cdot B) \otimes_B S
 = R\otimes_B S \approx E,
\end{align*}
hence $E\sse G$, which completes the proof.
\end{proof}

Following Williams \cite{Wil74} we denote by $\tsse$ the \emph{transitive closure of the relation $\sse$}. That is, $E \tsse F$ if there are $n$ $\ca$-correspondences $T_i$, $i=0,\dots,n$, such that $T_0=E$, $T_n= F$ and $T_i \sse T_{i+1}$.

There is also another relation between $\ca$-correspondences inspired by Williams' work \cite{Wil74}.

%%%%%%%%%%%%%%%%%%%%%%%%%%%%%%%%
\begin{definition}\label{D:se}
Let ${}_A E_A$ and ${}_B F_B$ be $\ca$-correspondences. Then $E$ and $F$ will be called \emph{shift equivalent} (symb.  $E\se F$) if there are $\ca$-correspondences ${}_A R_B$ and ${}_B S_A$ and a natural number $m$ so that
\begin{align*}
&\text{\textup{(i)}}\quad  E^{\otimes m} \approx R \otimes_A S ,\ F^{\otimes m} \approx  S \otimes_B R,\\
&\text{\textup{(ii)}}\quad S\otimes_A E\approx  F \otimes_B S,\ E\otimes_A R
\approx  R \otimes_B F.
\end{align*}
The number $m$ is called the \emph{lag} of the equivalence.
\end{definition}

If $E \se F$ with lag $m$, then by replacing $S$ with $S\otimes_A E^{\otimes k}$ we obtain another shift equivalence of lag $m+k$, i.e., there is a shift equivalence of lag $L$ for every $L \geq m$.

In contrast to elementary strong shift equivalence, shift equivalence is an equivalence relation  when restricted to subclasses that are closed under tensoring. In fact it is an equivalence relation for the whole class of arbitrary $\ca$-correspondences.

%%%%%%%%%%%%%%%%%%%%%%%%%%%%%%%%
\begin{proposition}\label{P:se is trans}
Shift equivalence is an equivalence relation.
\end{proposition}
\begin{proof}
Note that $E \se E$ with lag $2$ and $R=S=E$, and it is clear that shift equivalence is symmetric. If $E \se F$ with lag $m$ via $R, S$ and $F \se G$ with lag $n$ via $V, U$, then $E\se G$ with lag $mn+m$ via
\begin{align*}
\underbrace{R \otimes V \otimes U \otimes \cdots \otimes V}_{V \text{ is repeated } m \text{ times}}\, \text{ and } \,  U\otimes S,
\end{align*}
which shows that $\se$ is transitive.
\end{proof}

%%%%%%%%%%%%%%%%%%%%%%%%%%%%%%%%
\begin{theorem}\label{T:main 1}
Let ${}_A E_A$ and ${}_B F_B$ be $\ca$-correspondences. Then
\begin{align*}
E \sme F \, \Rightarrow \, E \sse F \,  \Rightarrow \, E\tsse F \, \Rightarrow \,  E \se F.
\end{align*}
\end{theorem}
\begin{proof}
Recall that when $E\sme F$ then $E,F$ are non-degenerate. Let $R\equiv E\otimes M$ and $S \equiv M^*$. By Lemma \ref{L:technical},
\begin{align*}
R\otimes_B S
& \approx  E \otimes_A M \otimes_{B} M^*
\approx  E\otimes_A A \approx E,
\end{align*}
and
\begin{align*}
S \otimes_A R & \approx M^* \otimes_{A} E\otimes_A M
\approx  M^* \otimes_A M\otimes_B F
\approx  B\otimes_B F \approx F.
\end{align*}
Hence $\sme$ implies $\sse$. Trivially $\sse$ implies $\tsse$. To complete the proof assume that $E\tsse F$, i.e., there are $T_i$, $i=0,\dots,n$, such that $T_i \sse T_{i+1}$, where $T_0=E$ and $T_n= F$. Then one can directly verify that $E \se F$ with lag $n$, via $R= R_1 \otimes \dots \otimes R_n$ and $S= S_n \otimes \dots \otimes S_1$.
\end{proof}

Muhly, Pask and Tomforde \cite{MPT08} provide a number of examples to show that strong Morita equivalence differs from elementary strong shift equivalence. In Theorem \ref{T:main 1} above we prove that in fact it is stronger. Nevertheless, the following result shows that under certain circumstances, the two notions coincide.

%%%%%%%%%%%%%%%%%%%%%%%%%%%%%%%%
\begin{proposition}\label{P:sse gives sme}
Let ${}_A E_A$ and ${}_B F_B$ be $\ca$-correspondences and assume  that $E \sse F$ via $R, S$. If either $R$ or $S$ is an imprimitivity bimodule, then $E\sme F$.
\end{proposition}
\begin{proof}
When $E\sse F$ via $R, S$ then $E \otimes_A R  \approx R\otimes_B S \otimes_A R \approx R\otimes_B F$ (analogously $S\otimes_B E \approx F \otimes_B S$).
\end{proof}

%%%%%%%%%%%%%%%%%%%%%%%%%%%%%%%%
\section{Passing to Pimsner Dilations} \label{S:equiv}
%%%%%%%%%%%%%%%%%%%%%%%%%%%%%%%%

In this section we show that if $\sim$ is any of the four relations defined in the previous section, then $E \sim F$ implies $E_\infty \sim F_\infty$, under the standing hypothesis that $E$ and $F$ are regular $\ca$-correspondences. Our study is based on the concept of a \textit{bipartite inflation}, an insightful construct originating in the work of Muhly, Pask and Tomforde \cite{MPT08}.

Assume that ${}_A E_A$ and ${}_B F_B$ are $\ca$-correspondences which are elementary strong shift equivalent via  ${}_A R_B$ and ${}_B S_A$. Let $X= \left[\begin{array}{c|c} 0 & R\\ S & 0 \end{array}\right]$ be \emph{the bipartite inflation of $S$ by $R$}. By Lemma \ref{L:technical}, we obtain\footnote{\ In order to ease notation, unitary equivalence of $\ca$-correspondences will be simply denoted as equality in this section.}
\[
X^{\otimes 2}= \left[\begin{array}{c|c} R\otimes_B S & 0\\ 0 & S\otimes_A R \end{array}\right] = \left[\begin{array}{c|c} E & 0\\ 0 & F \end{array}\right].
\]
By induction and Lemma \ref{L:technical},
\begin{align*}\label{eq: powers}
X^{\otimes 2k}= \left[\begin{array}{c|c} E^{\otimes k} & 0\\ 0 &
F^{\otimes k}
\end{array}\right], \quad
X^{\otimes 2k +1} = \left[\begin{array}{c|c}
0& E^{\otimes k}\otimes_A R\\
F^{\otimes k} \otimes_B S & 0
\end{array}\right],
\end{align*}
for all $k \in \bbZ_+$. In particular, if $E$ and $F$ are regular (resp. non-degenerate) then $R$ and $S$, and consequently $X$ and $X^{\otimes 2}$, are regular (resp. non-degenerate) as shown in \cite{MPT08}.

%%%%%%%%%%%%%%%%%%%%%%%%%%%%%%%%
\begin{proposition}\label{P:technical}
Let $E \sse F$ via $R,S$. Then, for $k\in \bbZ_+$,
\begin{enumerate}
\item $\sca{ E^{\otimes k}\otimes R, E^{\otimes k}\otimes R }$ provides a right c.a.i. for $F^{\otimes k+1}$,
\item $\sca{ F^{\otimes k}\otimes S, F^{\otimes k}\otimes S }$ provides a right c.a.i. for $E^{\otimes k+1}$.
\end{enumerate}
\end{proposition}
\begin{proof}
Let $X$ be the bipartite inflation of $S$ by $R$ and let $k \in \bbZ_+$. By Lemma \ref{L:cai ten pr} we have that $\sca{X^{\otimes 2k +1},X^{\otimes 2k +1}}$ provides a right c.a.i. for $X^{\otimes 2k+2}=X \otimes X^{\otimes 2k+1}$. However,
\[
X^{\otimes 2k+2}= \left[\begin{array}{c|c} E^{\otimes
k+1} & 0\\ 0 & F^{\otimes k+1}
\end{array}\right]
\]
and
\begin{align*}
&\sca{X^{\otimes 2k +1},X^{\otimes 2k +1}}= \\
&= \sca{\left[\begin{array}{c|c}
0& E^{\otimes k}\otimes_A R\\
F^{\otimes k} \otimes_B S & 0
\end{array}\right], \left[\begin{array}{c|c}
0& E^{\otimes k}\otimes_A R\\
F^{\otimes k} \otimes_B S & 0
\end{array}\right]}\\
&= \sca{F^{\otimes k}\otimes S, F^{\otimes k}\otimes
S}\oplus \sca{E^{\otimes k}\otimes R, E^{\otimes k}\otimes R}.
\end{align*}
This completes the  proof.
\end{proof}

We will make use of the dilation $X_\infty$ of the bipartite inflation of $S$ by $R$. Let $(A\oplus B)_\infty$ be the direct limit $\ca$-algebra of the following directed system
{\footnotesize
\begin{align*}
\left[\begin{array}{c|c} A & 0\\ 0 & B
\end{array}\right] & \stackrel{\rho_0}{\longrightarrow}
\K\left( \left[\begin{array}{c|c} 0 & R\\ S & 0
\end{array}\right] \right) \stackrel{\rho_1}{\longrightarrow}
\K\left( \left[\begin{array}{c|c} E & 0\\ 0 & F
\end{array}\right] \right) \stackrel{\rho_2}{\longrightarrow} \\
& \stackrel{\rho_2}{\longrightarrow} \K\left(
\left[\begin{array}{c|c} 0 & E \otimes R\\ F \otimes S & 0
\end{array}\right] \right) \stackrel{\rho_3}{\longrightarrow}
\K\left( \left[\begin{array}{c|c} E^{\otimes 2} & 0\\ 0 & F^{\otimes
2}
\end{array}\right] \right) \stackrel{\rho_4}{\longrightarrow}
\cdots,
\end{align*}
}where
\begin{align*}
&\rho_0=\phi_X \colon A=\L(A)\longrightarrow \L(X),\\
&\rho_n=\otimes \id_X \colon \L(X^{\otimes n}) \longrightarrow \L(X^{\otimes
n+1})\colon r \longmapsto r\otimes \id_X , \, n\geq 1.
\end{align*}
Note that, since $X$ and $X^{\otimes 2}$ are regular, $(A \oplus B)_\infty$ is $*$-isomorphic to the direct limit $\ca$-algebra of the following directed system
{\footnotesize
\begin{align*}
&\K\left( \left[\begin{array}{c|c} A & 0\\ 0 & B
\end{array}\right] \right) \stackrel{\rho_1\circ \rho_0}{\longrightarrow}
\K\left( \left[\begin{array}{c|c} E & 0\\ 0 & F
\end{array}\right] \right) \stackrel{\rho_3\circ \rho_2}{\longrightarrow}
\K\left( \left[\begin{array}{c|c} E^{\otimes 2} & 0\\ 0 & F^{\otimes
2}
\end{array}\right] \right) \stackrel{\rho_5\circ\rho_4}{\longrightarrow}
\cdots.
\end{align*}
}Since $E$ and $F$ are orthogonal subcorrespondences of $X^{\otimes 2}$ we get that $(A\oplus B)_\infty= A_\infty \oplus B_\infty$ and we can write $A_\infty$ and $B_\infty$ via the following directed systems
{\footnotesize
\begin{align*}
& \K\left( \left[\begin{array}{c|c} A & 0\\ 0 & 0
\end{array}\right] \right) \stackrel{\rho_1\circ \rho_0}{\longrightarrow}
\K\left( \left[\begin{array}{c|c} E & 0\\ 0 & 0
\end{array}\right] \right) \stackrel{\rho_3\circ \rho_2}{\longrightarrow}
\K\left( \left[\begin{array}{c|c} E^{\otimes 2} & 0\\ 0 & 0
\end{array}\right] \right) \stackrel{\rho_5\circ\rho_4}{\longrightarrow}
\cdots \longrightarrow A_\infty,\\
& \K\left( \left[\begin{array}{c|c} 0 & 0\\ 0 & B
\end{array}\right] \right) \stackrel{\rho_1\circ \rho_0}{\longrightarrow}
\K\left( \left[\begin{array}{c|c} 0 & 0\\ 0 & F
\end{array}\right] \right) \stackrel{\rho_3\circ \rho_2}{\longrightarrow}
\K\left( \left[\begin{array}{c|c} 0 & 0\\ 0 & F^{\otimes 2}
\end{array}\right] \right) \stackrel{\rho_5\circ\rho_4}{\longrightarrow}
\cdots \longrightarrow B_\infty.
\end{align*}
}
\noindent The $\ca$-correspondence $X_\infty$ is defined by the directed system
{\footnotesize
\begin{align*}
\left[\begin{array}{c|c} 0 & R\\ S & 0
\end{array}\right]
& \stackrel{\tau}{\longrightarrow}
\K\left( \left[\begin{array}{c|c} 0 & R\\ S & 0
\end{array}\right], \left[\begin{array}{c|c} E & 0\\ 0 & F
\end{array}\right] \right) \stackrel{\sigma_1}{\longrightarrow}\\
& \stackrel{\sigma_1}{\longrightarrow}
\K\left(\left[\begin{array}{c|c} E & 0\\ 0 & F
\end{array}\right] ,\left[\begin{array}{c|c} 0 & E\otimes R\\ F \otimes S & 0
\end{array}\right] \right) \stackrel{\sigma_2}{\longrightarrow}\\
& \stackrel{\sigma_2}{\longrightarrow}
\K\left(\left[\begin{array}{c|c} 0 & E\otimes R\\ F \otimes S & 0
\end{array}\right], \left[\begin{array}{c|c} E^{\otimes 2} & 0\\ 0 & F^{\otimes 2}
\end{array}\right]  \right)
\stackrel{\sigma_3}{\longrightarrow} \cdots
\end{align*}
}
where $\tau_\xi(\eta) = \xi \otimes \eta$, and {\small
\begin{align*}
\sigma_n=\otimes \id_X \colon \L(X^{\otimes n}, X^{\otimes n+1}) \rightarrow
\L(X^{\otimes n+1}, X^{\otimes n+2}) \colon  s \mapsto s\otimes
\id_X , \, n\geq 1.
\end{align*}
}
Let $R_\infty$ be the subcorrespondence of $X_\infty$ that is generated by the copies of
\begin{align*}
R\, , \K\left(\left[\begin{array}{c|c} 0 & 0\\
0 & F^{\otimes k}
\end{array}\right] ,\left[\begin{array}{c|c} 0 & E^{\otimes k} \otimes R\\ 0 &
0 \end{array}\right] \right),
\end{align*}
and
\begin{align*}
\K\left(\left[\begin{array}{c|c} 0 & 0\\ S & 0
\end{array}\right] ,\left[\begin{array}{c|c} E & 0\\ 0 &
0 \end{array}\right] \right), \,
\K\left(\left[\begin{array}{c|c} 0 & 0\\ F^{\otimes k}\otimes S & 0
\end{array}\right] ,\left[\begin{array}{c|c} E^{\otimes k+1} & 0\\ 0 &
0 \end{array}\right] \right),
\end{align*}
for $k\geq 1$. Because of Lemma \ref{L:compact ten pr}, $R_\infty$ can be written alternatively as a direct limit in the following two forms
\begin{align*}
\text{(i)}  \ & R \stackrel{\tau \circ (\otimes \id_{X})}{\longrightarrow}
\K\left(F,  E\otimes R \right)\stackrel{\otimes \id_{X^{\otimes 2}}}{\longrightarrow}
\K\left(F^{\otimes 2} , E^{\otimes 2}\otimes R \right)
\stackrel{\otimes \id_{X^{\otimes 2}}}{\longrightarrow} \cdots ,\\
\text{(ii)} \ & R \stackrel{\tau}{\longrightarrow}
\K\left(S , E  \right) \stackrel{\otimes \id_{X^{\otimes 2}}}{\longrightarrow}
\K\left(F \otimes S , E^{\otimes 2} \right)
\stackrel{\otimes \id_{X^{\otimes 2}}}{\longrightarrow} \cdots,
\end{align*}
where we omit the zero entries for convenience. The first form of $R_\infty$ shows that it is a Hilbert $B_\infty$-module, whereas the second form shows that the left multiplication by elements of $A_\infty$ defines a left action of $A_\infty$ on $R_\infty$. Hence $R_\infty$ is an $A_\infty$-$B_\infty$-correspondence.

In a dual way we define $S_\infty$ as the subcorrespondence of $X_\infty$ generated by the copies of
\begin{align*}
S \, , \K\left(\left[\begin{array}{c|c} E^{\otimes k} & 0\\
0 & 0
\end{array}\right] ,\left[\begin{array}{c|c} 0 & 0\\ 0 &
F^{\otimes k} \otimes S \end{array}\right] \right),
\end{align*}
and
\begin{align*}
\K\left(\left[\begin{array}{c|c} 0 & R\\ 0 & 0
\end{array}\right] ,\left[\begin{array}{c|c} 0 & 0\\ 0 &
F \end{array}\right] \right), \,
\K\left(\left[\begin{array}{c|c} 0 & E^{\otimes k}\otimes R\\ 0 & 0
\end{array}\right] ,\left[\begin{array}{c|c} 0 & 0\\ 0 &
F^{\otimes k+1} \end{array}\right] \right),
\end{align*}
for $k\geq 1$. By using Lemma \ref{L:compact ten pr} and writing $S_\infty$ as a direct limit in the following two forms
\begin{align*}
\text{(i)}  \ & S \stackrel{\tau \circ (\otimes \id_{X})}{\longrightarrow}
\K\left(E , F \otimes S \right) \stackrel{\otimes \id_{X^{\otimes 2}}}{\longrightarrow}
\K\left( E^{\otimes 2} , F^{\otimes 2}\otimes S  \right)
\stackrel{\otimes \id_{X^{\otimes 2}}}{\longrightarrow} \cdots ,\\
\text{(ii)} \ & S \stackrel{\tau}{\longrightarrow}
\K\left(R , F  \right) \stackrel{\otimes \id_{X^{\otimes 2}}}{\longrightarrow} \K\left( E
\otimes R, F^{\otimes 2} \right) \stackrel{\otimes \id_{X^{\otimes 2}}}{\longrightarrow} \cdots,
\end{align*}
we get that $S_\infty$ becomes a $B_\infty$-$A_\infty$-correspondence.

We must remark here on our use of the subscript $\infty$. For the correspondences  $E, F$ we denote by $E_\infty, F_\infty$ their Pimsner dilations, whereas for $R,S$ we denote by $R_\infty, S_\infty$ the subcorrespondences in the Pimsner dilation $X_\infty$ of the (injective) bipartite inflation $X$ of $S$ by $R$. Nevertheless, when the correspondences are non-degenerate, $R_\infty$ is the $X$-dilation $(\K(B,R), \id_X)_\infty$ and $S_\infty$ is the $X$-dilation $(\K(A,S), \id_X)_\infty$.

%%%%%%%%%%%%%%%%%%%%%%%%%%%%%%%%
\begin{remark}\label{R:useful}
Note that $R_\infty$ and $S_\infty$ are both full left Hilbert bimodules (thus injective). In order to prove this  for, say $S_\infty$, it is enough (by its second form) to show that
\[
\overline{\K(E^{\otimes n} \otimes R, F^{\otimes n+1}) \K(E^{\otimes n} \otimes R, F^{\otimes n+1})^*} =\K( F^{\otimes n+1},  F^{\otimes n+1})
\]
for all $n \geq 1$. By Lemma \ref{L:cai for compact} it suffices to show that $\sca{E^{\otimes n} \otimes R, E^{\otimes n} \otimes R}$ contains a right c.a.i. for $F^{\otimes n+1}$, for all $n \geq 1$, a fact established in Proposition \ref{P:technical}. A symmetrical argument can be used for $R_\infty$.

Moreover, if $S$ is an imprimitivity bimodule, then $S_\infty$ is also right full, hence an imprimitivity bimodule as well. Indeed, by the first form of $S_\infty$ it suffices to show that
\[
\overline{\K(E^{\otimes n}, F^{\otimes n} \otimes S)^* \K(E^{\otimes n}, F^{\otimes n} \otimes S)}=\K(E^{\otimes n}, E^{\otimes n}),
\]
for all $n$. By Lemma \ref{L:cai for compact}, it is enough to show that $\sca{F^{\otimes n} \otimes S, F^{\otimes n}\otimes S}$ provides a right c.a.i. for $E^{\otimes n}$. Note that $F\otimes S = S \otimes E$, hence $F^{\otimes n} \otimes S = S \otimes E^{\otimes n}$, thus $E^{\otimes n}$ is non-degenerate. Therefore
\begin{align*}
\sca{F^{\otimes n} \otimes S, F^{\otimes n}\otimes S}
& = \sca{S \otimes E^{\otimes n}, S \otimes E^{\otimes n}}\\
& = \sca{E^{\otimes n}, \phi_{E^{\otimes n}}(\sca{S,S})E^{\otimes n}}\\
& = \sca{E^{\otimes n}, \phi_{E^{\otimes n}}(A)E^{\otimes n}} \\
& = \sca{E^{\otimes n}, E^{\otimes n}},
\end{align*}
and the latter ideal provides a right c.a.i. for $E^{\otimes n}$.
\end{remark}

%%%%%%%%%%%%%%%%%%%%%%%%%%%%%%%%
\begin{theorem}\label{T:main 2}
Let ${}_A E_A, {}_B F_B$ be regular $\ca$-correspondences. If $E \sse F$, then $E_\infty \sse F_\infty$.
\end{theorem}
\begin{proof}
Assume that $E, F$ are elementary strong shift equivalent via $R, S$. It suffices to prove that the interior tensor product $S_\infty \otimes_{A_\infty} R_\infty$ is (unitarily equivalent to) $F_\infty$. Then, by duality, $R_\infty \otimes_{B_\infty} S_\infty$ is (unitarily equivalent to) $E_\infty$, hence $E_\infty \sse F_\infty$. Towards this end we view $S_\infty$ as the $X^{\otimes 2}$-dilation of $\K(E\otimes R, F^{\otimes 2})$, i.e.,
\begin{align*}
\K\left( E\otimes R, F^{\otimes 2} \right) \stackrel{\otimes \id_{X^{\otimes
2}}}{\longrightarrow} \K\left( E^{\otimes 2}\otimes R, F^{\otimes 3}
\right) \stackrel{\otimes \id_{X^{\otimes 2}}}{\longrightarrow} \cdots,
\end{align*}
and $R_\infty$ as the $X^{\otimes 2}$-dilation of $\K(F, E\otimes R)$, i.e.,
\begin{align*}
\K\left( F , E\otimes R \right) \stackrel{\otimes \id_{X^{\otimes
2}}}{\longrightarrow} \K\left( F^{\otimes 2},  E^{\otimes 2}\otimes
R \right) \stackrel{\otimes \id_{X^{\otimes 2}}}{\longrightarrow} \cdots .
\end{align*}
Note that $A_\infty$ can be written as $(\K(E\otimes R), \id_{X^{\otimes 2}})_\infty$. By Proposition \ref{P:technical}, the ideal $\sca{E^{\otimes k} \otimes R, E^{\otimes k} \otimes R}$ provides a right c.a.i. for $F^{\otimes k+1}$, for every $k\geq 0$. Therefore, $\sca{E\otimes R \otimes (X^{\otimes 2})^{\otimes n},E\otimes R \otimes (X^{\otimes 2})^{\otimes n}}$ provides a right c.a.i. for $F^{\otimes 2} \otimes (X^{\otimes 2})^{\otimes n}$, for every $n\geq 0$. Thus, Proposition \ref{P:tens reg shifts} applies, and we obtain that $S_\infty \otimes_{A_\infty} R_\infty$ is $(\K(F, F^{\otimes 2}), \id_{X^{\otimes 2}})_\infty = F_\infty$.
\end{proof}

%%%%%%%%%%%%%%%%%%%%%%%%%%%%%%%%
\begin{remark}
In the above proof we have actually shown that
\begin{align*}
(R\otimes_B S)_\infty = E_\infty =  R_\infty \otimes_{B_\infty} S_\infty,
\end{align*}
and that
\begin{align*}
(S\otimes_A R)_\infty = F_\infty =  S_\infty \otimes_{A_\infty} R_\infty.
\end{align*}
\end{remark}

An immediate consequence of Theorem \ref{T:main 2} is the following.

%%%%%%%%%%%%%%%%%%%%%%%%%%%%%%%%
\begin{theorem}
Let ${}_A E_A, {}_B F_B$ be regular $\ca$-correspondences. If $E\tsse F$, then $E_\infty \tsse F_\infty$.
\end{theorem}
\begin{proof}
Assume that $E \tsse F$ via a sequence of $T_i$, for $i=0,\dots,n$. Then $T_i \sse T_{i+1}$, for $i=0,\dots,n$. By Theorem \ref{T:main 2}, we get that $(T_i)_\infty \sse (T_{i+1})_\infty$ for $i=0,\dots,n$. Since $T_0=E$ and $T_n=F$, then $E_\infty \tsse F_\infty$.
\end{proof}

We use Theorem \ref{T:main 2} to understand the passage to dilations for the strong Morita equivalence.

\begin{theorem}\label{T:sme infty}
Let ${}_A E_A, {}_B F_B$ be regular $\ca$-correspondences. If $E\sme F$, then $E_\infty \sme F_\infty$.
\end{theorem}
\begin{proof}
Assume that $E \sme F$; then $E$ and $F$ are non-degenerate and $E \sse F$ via $E\otimes_A M^*$ and $M$ (see Theorem \ref{T:main 1}). Therefore, by Theorem \ref{T:main 2}, $E_\infty \sse F_\infty$ via $(E\otimes M^*)_\infty$ and $M_\infty$. Since $M$ is assumed an imprimitivity bimodule, then $M_\infty$ is also an imprimitivity bimodule by Remark \ref{R:useful}. By Proposition \ref{P:sse gives sme} we conclude that $E_\infty \sme F_\infty$.
\end{proof}

%%%%%%%%%%%%%%%%%%%%%%%%%%%%%%%%
\begin{remark}
An alternative way of attacking Theorem \ref{T:sme infty} appears in \cite[Remark 3.6]{MuhSol00} where Muhly and Solel propose that, when $E \sme F$ via an imprimitivity bimodule $X$, then $E_\infty \sme F_\infty$ via the imprimitivity bimodule $\fX= A_\infty \otimes_A X \otimes_B B_\infty$. It may take no effort to show that $E\otimes_A A_\infty = A_\infty \otimes_A E$, as sets, and conclude that $E_\infty \otimes_{A_\infty} \fX = \fX \otimes_{B_\infty} F_\infty$. However there is a problem with the definition of $\fX$. It is easy to define a left and right action of $A$ on $A_\infty$ (simply by multiplication), but in order to get the tensor product $A_\infty \otimes_A X$ (and so $\fX$) an inner product of $A_\infty$ taking values in $A$ is needed. The existence of such an inner product is not obvious to us.
\end{remark}

We have arrived to the last relation to be examined when passing to Pimsner dilations. Our next Theorem is one of the central results in this paper and will enable us to obtain new information even for concrete classes of operator algebras, i.e., Cuntz-Krieger $\ca$-algebras.

%%%%%%%%%%%%%%%%%%%%%%%%%%%%%%%%
\begin{theorem}\label{T:se infty}
Let ${}_A E_A, {}_B F_B$ be regular $\ca$-correspondences. If $E \se F$ with lag $m$, then $E_\infty \se F_\infty$ with lag $m$.
\end{theorem}
\begin{proof}
Let $R, S$ such that
\begin{align*}
&\text{\textup{(i)}}\quad  E^{\otimes m}=R \otimes_A S, \, F^{\otimes
m}=S
\otimes_A R,\\
&\text{\textup{(ii)}}\quad S\otimes_A E= F \otimes_B S, \, E\otimes_A R =
R \otimes_B F.
\end{align*}
Thus $E^{\otimes m} \sse F^{\otimes m}$. Since $E$ is regular, we have
\[
E_\infty= (\K(E,E^{\otimes 2}),\id_E)_\infty =  (\K(E^{\otimes 2}, E^{\otimes 3}),\id_E)_\infty.
\]
Proposition \ref{P:tens reg shifts} implies that
\begin{align*}
(E_\infty)^{\otimes 2} 
& = (\K(E^{\otimes 2},E^{\otimes 3}),\id_E)_\infty \otimes_{A_\infty} (\K(E, E^{\otimes 2}),\id_E)_\infty \\
& = (\K(E, E^{\otimes 3}),\id_E)_\infty \\
& = (E^{\otimes 2})_\infty.
\end{align*}
A repetitive use of this argument shows that $(E^{\otimes m})_\infty= (E^{\otimes m-1})_\infty \otimes E_\infty = (E_\infty)^{\otimes m}$. Thus, by the remark following Theorem \ref{T:main 2}, we obtain
\begin{align*}
(E_{\infty})^{\otimes m}= (E^{\otimes m})_\infty = (R\otimes_B S)_\infty= R_\infty
\otimes_{B_\infty} S_\infty ,
\end{align*}
and in a similar fashion $(F_\infty)^{\otimes m}= S_\infty \otimes_{A_\infty} R_\infty$.

What remains to be proved, in order to complete the proof, is that
\begin{align*}
E_\infty \otimes_{A_\infty} R_\infty = R_\infty
\otimes_{B_\infty} F_\infty.
\end{align*}
Since $E$ is regular, we get that $E_\infty$ coincides with
$(\K(E^{\otimes m}, E^{\otimes m+1}), \id_{E})_\infty= (\K(E^{\otimes m}, E^{\otimes m+1}), \id_{X^{\otimes 2}})_\infty$, so
\begin{align*}
E_\infty = (\K(E^{\otimes m}, E^{\otimes m+1}), \id_{X^{\otimes 2}})_\infty \, ,
R_\infty = (\K(S, E^{\otimes m}), \id_{X^{\otimes 2}})_\infty,
\end{align*}
where $X$ is the bipartite inflation of $R$ and $S$.

Now $E^{\otimes m+ mn}= E^{\otimes m} \otimes (X^{\otimes 2})^n$, $n\geq 0$, and, by Lemma \ref{L:cai ten pr}, the ideal $\sca{E^{\otimes m+ mn}, E^{\otimes m+ mn}}= \sca{E^{\otimes m}\otimes (X^{\otimes 2})^n, E^{\otimes m}\otimes (X^{\otimes 2})^n}$ provides a right c.a.i. for $E \otimes E^{\otimes m+ mn}= E^{\otimes 1+ m+ mn}=  E^{\otimes m+1} \otimes (X^{\otimes 2})^n$. Hence Proposition \ref{P:tens reg shifts} implies that
\begin{align*}
E_\infty \otimes_{A_\infty} R_\infty = (\K(S, E^{\otimes
m+1}), \id_{X^{\otimes 2}})_\infty.
\end{align*}

Similarly, we express
\begin{align*}
R_\infty= (\K(F^{\otimes m}, E^{\otimes m}\otimes R), \id_{X^{\otimes 2}})_\infty \, , F_\infty = (\K(F^{\otimes m-1}, F^{\otimes m}), \id_{X^{\otimes 2}})_\infty;
\end{align*}
then
\begin{align*}
R_\infty \otimes_{B_\infty} F_\infty = (\K(F^{\otimes m-1},
E^{\otimes m} \otimes R), \id_{X^{\otimes 2}})_\infty,
\end{align*}
since $F^{\otimes m+mn}= F^{\otimes m} \otimes (X^{\otimes 2})^n$ and $\sca{F^{\otimes m+mn}, F^{\otimes m+mn}}$ provides a right c.a.i. for $R \otimes F^{\otimes m+mn} = E^{\otimes m} \otimes R \otimes F^{\otimes mn} = E^{\otimes m} \otimes R \otimes (X^{\otimes 2})^n$.

To prove that $E_\infty \otimes R_\infty = R_\infty \otimes F_\infty$ we show that each one is unitarily equivalent to a submodule of the other. First we observe that for $U= \left[ \begin{array}{c|c} 0 & E\otimes R\\ F \otimes S & 0 \end{array} \right]$ we get
\begin{align*}
X \otimes U
& =
\left[ \begin{array}{c|c} 0 &  R\\ S & 0 \end{array} \right] \otimes
\left[ \begin{array}{c|c} 0 & E\otimes R\\ F \otimes S & 0 \end{array} \right]\\
& =
\left[ \begin{array}{c|c} R \otimes F \otimes S & 0\\  0 & S \otimes E \otimes R  \end{array} \right] \\
& =
\left[ \begin{array}{c|c} E \otimes R \otimes S & 0\\  0 & F \otimes S \otimes R  \end{array} \right] \\
& =
U \otimes X,
\end{align*}
Thus $X^{\otimes 2} \otimes U = U \otimes X^{\otimes 2}$ and, by Proposition \ref{P:wz=zw}, we have that the correspondence
\begin{align*}
R_\infty \otimes_{B_\infty} F_\infty= (\K(F^{\otimes m-1}, E^{\otimes m} \otimes R), \id_{X^{\otimes 2}})_\infty
\end{align*}
is unitarily equivalent, via some unitary $u$, to a submodule of
\begin{align*}
& (\K(F^{\otimes m-1}\otimes U, E^{\otimes m} \otimes R\otimes U),
\id_{X^{\otimes 2}})_\infty=\\
& \qquad \qquad \qquad =(\K(F^{\otimes m-1} \otimes F\otimes S, E^{\otimes m+1}\otimes R\otimes S), \id_{X^{\otimes 2}})_\infty\\
& \qquad \qquad \qquad =(\K(S\otimes R\otimes S, E^{\otimes m+1}\otimes R\otimes S), \id_{X^{\otimes 2}})_\infty\\
& \qquad \qquad \qquad =(\K(S\otimes X^{\otimes 2}, E^{\otimes m+1}\otimes X^{\otimes 2}), \id_{X^{\otimes
2}})_\infty\\
& \qquad \qquad \qquad =(\K(S, E^{\otimes m+1}), \id_{X^{\otimes
2}})_\infty\\
& \qquad \qquad \qquad = E_\infty \otimes R_\infty.
\end{align*}
Also define $V= \left[ \begin{array}{c|c} 0 &  E^{\otimes m-1}\otimes R\\ F^{\otimes m-1}\otimes S & 0 \end{array} \right]$ and verify that $X\otimes V= V \otimes X$. Thus, again by Proposition \ref{P:wz=zw}, the correspondence
\begin{align*}
E_\infty \otimes R_\infty= (\K(S, E^{\otimes m+1}), \id_{X^{\otimes
2}})_\infty
\end{align*}
is unitarily equivalent, via some unitary $v$, to a submodule of
\begin{align*}
& (\K(S\otimes V, E^{\otimes m+1}\otimes V), \id_{X^{\otimes 2}})=\\
& \qquad \qquad \qquad =(\K(F^{\otimes 2m-1}, E^{\otimes 2m}\otimes R), \id_{X^{\otimes
2}})_\infty\\
& \qquad \qquad \qquad =(\K(F^{\otimes m-1} \otimes X^{\otimes 2}, E^{\otimes m} \otimes X^{\otimes 2} \otimes R), \id_{X^{\otimes
2}})_\infty\\
& \qquad \qquad \qquad =(\K(F^{\otimes m-1}\otimes X^{\otimes 2}, E^{\otimes m}\otimes R \otimes X^{\otimes 2}), \id_{X^{\otimes
2}})_\infty\\
& \qquad \qquad \qquad= (\K(F^{\otimes m-1}, E^{\otimes m}\otimes R), \id_{X^{\otimes
2}})_\infty\\
& \qquad \qquad \qquad= R_\infty \otimes F_\infty.
\end{align*}
It remains to show that $uv=\id$ and $vu=\id$. Note that $uv$ maps an element $[q]$ of any compact shift considered above (either it is in a $\ca$-algebra, either it is an element of the module) to $[q\otimes \id_{U\otimes V}]= [q\otimes \id_{X^{\otimes 4}}]$, since
\begin{align*}
U\otimes V
& =
\left[ \begin{array}{c|c} 0 & E\otimes R\\ F \otimes S & 0 \end{array} \right]
\otimes
\left[ \begin{array}{c|c} 0 &  E^{\otimes m-1}\otimes R\\ F^{\otimes m-1}\otimes S & 0 \end{array} \right]\\
& =
\left[ \begin{array}{c|c}  E\otimes R \otimes F^{\otimes m-1}\otimes S & 0\\ 0 & F \otimes S\otimes E^{\otimes m-1} \otimes R \end{array} \right]\\
& =
\left[\begin{array}{c|c}  E\otimes E^{\otimes m-1} \otimes R \otimes S & 0\\ 0 & F \otimes F^{\otimes m-1} \otimes S \otimes R \end{array}\right]\\
& =
\left[ \begin{array}{c|c}  E^{\otimes m} \otimes E^{\otimes m} & 0\\ 0 & F^{\otimes m} \otimes F^{\otimes m} \end{array}\right] \\
& = 
X^{\otimes 4}.
\end{align*}
However, throughout the proof we  have used exclusively $X^{\otimes 2}$-dilations and so $[q\otimes \id_{X^{\otimes 4}}]=[q]$. Therefore $uv=\id$. In a similar way $vu=\id$, since $V\otimes U= X^{\otimes 4}$, and the proof is complete.
\end{proof}

As mentioned in the Introduction, we were strongly motivated by \cite[Remark 5.5]{MPT08} and one of our aims was to check whether the alternative way proposed there to prove \cite[Theorem 3.14]{MPT08} could be achieved, i.e., to show that if $E \sse F$ then $E_\infty \sme F_\infty$. There is, though, a delicate point in this approach. In \cite[Remark 5.5]{MPT08} the authors claim that if $E$ is a non-degenerate and regular $\ca$-correspondence, then the Pimsner dilation $E_\infty$ \cite{Pim97} is an imprimitivity bimodule. This is false in general as we show in Remark \ref{R:not imp}. In fact it is true only in the context of \cite{Pim97}, because the $\ca$-correspondences there are always assumed full \cite[Remark 1.2(3)]{Pim97}. This is a consequence of another delicate point in Pimsner's theory, as his version of $\ca$-algebras, unfortunately denoted by the same symbols $\T_E$ and $\O_E$, are not what have eventually become the usual $\ca$-algebras generated by the images of $X$ and $A$, but they are only generated by the image of $X$; hence there is no reason to make a distinction between full and non-full correspondences in Pimsner's theory. Note that when $X$ is regular one can recover $A$ in Pimsner's $\ca$-algebra $\O_E$, but this is not the case for $\T_E$.

Of course, if one adds the full-ness condition, then the scheme described in \cite[Remark 5.5]{MPT08} can be implemented, as we are about to show. Note that the previous discussion and the next result settle Conjecture 1 appearing in the Introduction.

%%%%%%%%%%%%%%%%%%%%%%%%%%%%%%%%
\begin{theorem}\label{T:full}
Let ${}_A E_A$ and ${}_B F_B$ be full, non-degenerate and regular $\ca$- correspondences. Then the following scheme
\begin{align*}
\xymatrix{
E\sme F     \ar@{=>}[d] \ar@{=>}[r]
& E \sse F  \ar@{=>}[d] \ar@{=>}[r]
& E\tsse F  \ar@{=>}[d] \ar@{=>}[r]
& E\se F    \ar@{=>}[d] \\
E_\infty \sme F_\infty     \ar@{<=>}[r]
& E_\infty \sse F_\infty  \ar@{<=>}[r]
& E_\infty \tsse F_\infty  \ar@{<=>}[r]
& E_\infty \se F_\infty }
\end{align*}
holds.
\end{theorem}
\begin{proof}
By assumption $E$ and $F$ are full and non-degenerate. Therefore $E_\infty$ and $F_\infty$ are imprimitivity bimodules, and Theorem \ref{T:sse=se=sme} (that will follow) applies.
\end{proof}

An immediate consequence of Theorem \ref{T:full} and Theorem \ref{T:inj} is the following.

%%%%%%%%%%%%%%%%%%%%%%%%%%%%%%%%
\begin{theorem}\label{T:se gives me}
Let ${}_A E_A$ and ${}_B F_B$ be full, non-degenerate and regular $\ca$- correspondences. If $E \se F$, then the corresponding Cuntz-Pimsner algebras are strong Morita equivalent.
\end{theorem}
\begin{proof}
Suppose that $E \se F$. Then by Theorem \ref{T:full} we have that $E_\infty \sme F_\infty$. Therefore \cite[Theorem 3.5]{MuhSol00} implies that $\O_{E_\infty} \sme \O_{F_\infty}$ and the conclusion follows from Theorem \ref{T:inj}.
\end{proof}

In particular we obtain the following result for Cuntz-Krieger $\ca$-algebras mentioned in the Introduction.

%%%%%%%%%%%%%%%%%%%%%%%%%%%%%%%%
\begin{corollary} \label{graph}
Let $\G$ and $\G'$ be finite graphs with no sinks or sources and let $A_{\G}$ and $A_{\G'}$ be their adjacent matrices. If $A_{\G} \se A_{\G'}$, in the sense of Williams, then the Cuntz-Krieger $\ca$-algebras $\O_{\G}$ and $\O_{\G'}$ are strong Morita equivalent.
\end{corollary}

There is also an application to unital injective dynamical systems.

%%%%%%%%%%%%%%%%%%%%%%%%%%%%%%%%
\begin{corollary} \label{dyn sys}
Let $(A,\ga)$ and $(B,\be)$ be unital injective dynamical systems. If $X_\ga \se X_\be$, then $X_{\ga_\infty} \se Y_{\be_\infty}$ and the crossed products $A_\infty \rtimes_{\ga_\infty} \bbZ$ and $B\rtimes_{\be_\infty} \bbZ$ are strong Morita equivalent.
\end{corollary}

We close this section by settling Conjecture 2 of the Introduction. This conjecture asserts that the vertical arrows in Theorem \ref{T:full} are actually equivalences. We will use the following.

%%%%%%%%%%%%%%%%%%%%%%%%%%%%%%%%
\begin{theorem}\label{T:se not inv}
Let ${}_A E_A$ be a full, non-degenerate and regular $\ca$- correspondence. If  $E \se E_\infty$ then $E$ is an imprimitivity bimodule.
\end{theorem}
\begin{proof}
By assumption there exist non-degenerate, regular $\ca$- correspondences ${}_A R_{A_{\infty}}, {}_{A_{\infty}}S_A$ and a positive integer $m$ such that
\[
E^{\otimes m} =  R \otimes S, \,(E_\infty)^{\otimes m}=  (E^{\otimes m})_\infty =  S\otimes R
\]
and $E \otimes R =  R \otimes E_\infty$, $S \otimes E =  E_\infty \otimes S$. Then
\begin{align*}
A=  \sca{E^{\otimes m}, E^{\otimes m}} =  \sca{R\otimes S, R \otimes S} =  \sca{S, \sca{R,R} S} \subseteq \sca{S,S} \subseteq A,
\end{align*}
since $E^{\otimes m}$ is also full. Hence $S$ is full. Now, let $k \in \K(S)$. Then $k\otimes \id_R \in \K((E_\infty)^{\otimes m})$, and since $(E_\infty)^{\otimes m}$ is an imprimitivity bimodule, there is an $x\in A_\infty$ such that $k\otimes \id_R=\phi_{(E_\infty)^{\otimes m}}(x)= \phi_S(x)\otimes \id_R$. Thus $\phi_S(x)=k$, since $R$ is regular; therefore $A_\infty \simeq \K(S)$, hence $S$ is an imprimitivity bimodule. Going back to the definition of $E \se E_\infty$, we see that $S \otimes E = E_\infty \otimes S$, hence $E= S^* \otimes E_\infty \otimes S$. The fact that $S, S^*$ and $E_\infty$ are imprimitivity bimodules implies that $E$ is also an imprimitivity bimodule.
\end{proof}

%%%%%%%%%%%%%%%%%%%%%%%%%%%%%%%%
\begin{remark}\label{R:conj 2}
Let us see now why Conjecture 2 has a negative answer. Let $E$ be any full, non-degenerate and regular $\ca$-correspondence, which is not an imprimitivity bimodule. For example let $E= X_\ga$ be the $\ca$-correspondence associated to an injective $\ca$-dynamical system $(A,\ga)$, with $\ga$ not onto. If Conjecture 2 was true, then $E \se E_\infty$, since both $E$ and $E_\infty$ have unitarily equivalent Pimsner dilations. But then Theorem \ref{T:se not inv} would imply that $E$ is an imprimitivity bimodule, a contradiction.
\end{remark}

%%%%%%%%%%%%%%%%%%%%%%%%%%%%%%%%
\section{Imprimitivity Bimodules and the Shift Equivalence Problem}\label{S:app}
%%%%%%%%%%%%%%%%%%%%%%%%%%%%%%%%

Our work on shift equivalences suggests the following generalization of the Williams Problem in the context of $\ca$-correspondences.

\medskip

\noindent \textbf{The Shift Equivalence Problem for $\ca$-correspondences:} If $E, F$ are non-degenerate and regular $\ca$-correspondences, then does $E \se F$ imply that $E \tsse F$?

\medskip

One might be tempted to say that the work of Kim and Roush \cite{KimRou99} readily shows that the answer to the above conjecture is negative. However, it is not known to us whether two graph $\ca$-correspondences which fail to be strong shift equivalent \textit{via non-negative integral matrices} remain inequivalent if one considers arbitrary $\ca$-correspondences in order to implement the strong shift equivalence. In other words, we do not know the answer to the Williams Problem even for the class of graph $\ca$-correspondences. If it does have a positive answer then it will provide an alternative route for proving our Corollary \ref{graph}, but we conjecture that it doesn't in general.

Nevertheless, in our next result we show that Shift Equivalence Problem  has an affirmative answer for the class of imprimitivity bimodules. Recall that this result is essential for the proof of Theorem \ref{T:full}.

%%%%%%%%%%%%%%%%%%%%%%%%%%%%%%%%
\begin{theorem}\label{T:sse=se=sme}
Strong Morita equivalence, elementary strong shift equivalence, strong shift equivalence and shift equivalence are equivalent for the class of imprimitivity bimodules.
\end{theorem}
\begin{proof}
In view of Theorem \ref{T:main 1}, it suffices to show that if ${}_A E_A$ and ${}_B F_B$ are imprimitivity bimodules such that $E \se F$, then $E\sme F$. Assume that $E\se F$ via $R,S$ with lag $m$. Then $E^{\otimes m}=R\otimes S$ and $F^{\otimes m} = S \otimes R$. Since $E,F$ are imprimitivity bimodules, then $E^{\otimes m}$ and $F^{\otimes m}$ are also imprimitivity bimodules. Hence $S$ is regular and, due to $\se$, intertwines $E$ and $F$. It suffices to prove that $S$ is an imprimitivity bimodule. First, it is full right since
\begin{align*}
A = \sca{E^{\otimes m}, E^{\otimes m}}  = \sca{R \otimes S, R \otimes S} = \sca{S, \sca{R,R}\cdot S}
\subseteq \sca{S, S} \subseteq A,
\end{align*}
thus $\sca{S,S}= A$. In order to prove that it is full left, it suffices to prove that $\phi_S\colon B \rightarrow \L(S)$ is onto $\K(S)$ (since $\phi_S(B) \subseteq \K(S)$, by regularity of $S$). To this end, let $k\in \K(S)$. Then, due to regularity of $R$, we obtain that $k\otimes \id_R \in \K(S\otimes R)= \K(F^{\otimes m})$. Since $F^{\otimes m}$ is an imprimitivity bimodule, there is a $b\in B$ such that $\phi_{F^{\otimes m}}(b) = k\otimes \id_R$; in particular $\phi_S(b)\otimes \id_R = k\otimes \id_R$. Thus $\phi_S(b)=k$,
since $R$ is regular.
\end{proof}

Combining \cite[Theorem 3.2, Theorem 3.5]{MuhSol00} with Theorem \ref{T:sse=se=sme} we obtain the following corollary.

%%%%%%%%%%%%%%%%%%%%%%%%%%%%%%%%
\begin{corollary}
If ${}_A E_A, {}_B F_B$ are imprimitivity bimodules and $E \sme F, E\sse F, E\tsse F$ or $E \se F$, then the corresponding Toeplitz-Cuntz-Pimsner algebras and Cuntz-Pimsner algebras are strong Morita equivalent as $\ca$-algebras, and the corresponding tensor algebras are strong Morita equivalent in the sense of \cite{BMP00}.
\end{corollary}

%%%%%%%%%%%%%%%%%%%%%%%%%%%%%%%%
\section{Other Applications}
%%%%%%%%%%%%%%%%%%%%%%%%%%%%%%%%

%%%%%%%%%%%%%%%%%%%%%%%%%%%%%%%%
\subsection{Extension of \cite[Theorem 3.14]{MPT08}}
%%%%%%%%%%%%%%%%%%%%%%%%%%%%%%%%

A natural question raised in \cite{MPT08} was whether \cite[Theorem 3.14]{MPT08} is valid without the extra assumption of non-degeneracy. This can be established now with the theory we have developed in Section \ref{S:equiv}.
However, one cannot dispose of regularity, as mentioned explicitly in \cite{MPT08} and \cite[Example 5.4]{Bat02}.

%%%%%%%%%%%%%%%%%%%%%%%%%%%%%%%%
\begin{theorem}\label{T:main 6}
Let ${}_A E_A$ and ${}_B F_B$ be regular $\ca$-correspondences. If $E\tsse F$ then $\O_E \sme \O_F$.
\end{theorem}
\begin{proof}
First suppose that $E\sse F$. By Theorem \ref{T:main 2} we have $E_\infty \sse F_\infty$. But $E_\infty$ and $F_\infty$ are regular Hilbert bimodules hence they are non-degenerate. Therefore \cite[Theorem 3.14]{MPT08} implies $\O_{E_\infty} \sme \O_{F_{\infty}}$ and Theorem \ref{T:inj} completes the proof in the case of elementary strong shift equivalence.

Let $E\tsse F$ be via a sequence of $T_i$, $i=0,\dots, n$. Then  $\O_{T_i} \sme \O_{T_{i+1}}$, by the previous arguments, for every $i=0,\dots, n-1$. Strong Morita equivalence of $\ca$-algebras is transitive, hence $\O_E =\O_{T_0} \sme \O_{T_n}=\O_F$.
\end{proof}

%%%%%%%%%%%%%%%%%%%%%%%%%%%%%%%%%%%%%%%%%%%%%%%%
\subsection{\cite[Proposition 4.2]{MPT08} Revisited}\label{Ss:strict}
%%%%%%%%%%%%%%%%%%%%%%%%%%%%%%%%%%%%%%%%%%%%%%%%

The results in \cite{MPT08} and here, concerning strong Morita equivalence of the Cuntz-Pimsner algebras, can be generalized for degenerate correspondences over unital $\ca$- algebras because of \cite[Proposition 4.2]{MPT08}, i.e., if $X$ is a correspondence over a unital $\ca$-algebra $A$, then $\O_{X_{ess}}$ is a full corner of $\O_X$, where $X_{ess}:=\overline{\phi_X(A)X}$. In fact the proofs in \cite[Proposition 4.2]{MPT08} apply in general for strict correspondences $X$. The key observation is that, if $(a_i)$ is a c.a.i. in $A$, then $\phi_X(a_i)$ converges in the s*-topology to a projection, say $p$, in $\L(X)$. As a consequence $J_X= J_{X_{ess}}$ and, if $X$ is regular, then so is $X_{ess}$, in the same way as in \cite[Proposition 4.1]{MPT08}.

%%%%%%%%%%%%%%%%%%%%%%%%%%%%%%%%%%%%%%%%%%%%%%%%
\begin{theorem}\label{T:strict}
Let ${}_A X_A$ be a strict $\ca$-correspondence. Then $\O_{X_{ess}}$ is a full corner of $\O_X$.
\end{theorem}
\begin{proof}
Fix a covariant injective representation $(\pi,t)$ of $X$ (that admits a gauge action); then $(\pi,t|_{X_{ess}})$ is a covariant injective representation of $X_{ess}$ (that admits a gauge action). Let $P$ be the projection in $\M(\O_X)$ that is defined by $\lim_i \pi(a_i)$ for a c.a.i. $(a_i)$ of $A$; for example
\[
Pt(\xi):= \lim_i \pi(a_i) t(\xi)= t(\lim_i \phi(a_i)\xi)=t(p\xi), \foral \xi
\in X.
\]
Then, the rest of the proof goes as the proof of \cite[Proposition 4.2]{MPT08}.
\end{proof}

%%%%%%%%%%%%%%%%%%%%%%%%%%%%%%%%
\begin{acknow}
The first author would like to thank his wife Maritina for her endless support and understanding.
\end{acknow}

%%%%%%%%%%%%%%%%%%%%%%%%%%%%%%%%

\end{document}